\documentclass{article}

\usepackage[preprint, nonatbib]{neurips_2019}

\usepackage[utf8]{inputenc} 
\usepackage[T1]{fontenc}    
\usepackage{hyperref}       
\usepackage{url}            
\usepackage{booktabs}       
\usepackage{amsfonts}       
\usepackage{nicefrac}       
\usepackage{microtype}      

\usepackage{graphicx}
\usepackage{subfigure}

\usepackage{color}
\usepackage{algorithmic}
\usepackage{algorithm}

\usepackage{amsmath}
\usepackage{amssymb}
\usepackage{amsthm}

\newtheorem{proposition}{Proposition}[section]

\newtheorem{lemma}{Lemma}[section]

\newtheorem{definition}{Definition}[section]
\newtheorem{theorem}{Theorem}[section]
\theoremstyle{remark}

\newcommand{\sgn}{\text{sgn}}

\newcommand{\bbR}{\mathbb{R}}

\newcommand{\bbE}{\mathbb{E}}

\newcommand{\calD}{\mathcal{D}}

\newcommand{\diam}{\mathrm{diam}}

\newcommand{\calZ}{\mathcal{Z}}

\newcommand{\calF}{\mathcal{F}}
\newcommand{\calG}{\mathcal{G}}

\DeclareMathOperator*{\argmin}{arg\,min}
\DeclareMathOperator*{\argmax}{arg\,max}

\title{Stochastic In-Face Frank-Wolfe Methods for Non-Convex Optimization and Sparse Neural Network Training}

\author{%
Paul Grigas \ \ \ \ \ Alfonso Lobos \ \ \ \ \ Nathan Vermeersch\\
Department of Industrial Engineering and Operations Research, University of California, Berkeley\\
\texttt{\{pgrigas, alobos, nathan.vermeersch\}@berkeley.edu} \\
}

\begin{document}

\maketitle


\begin{abstract}
The Frank-Wolfe method and its extensions are well-suited for delivering solutions with desirable structural properties, such as sparsity or low-rank structure. We introduce a new variant of the Frank-Wolfe method that combines Frank-Wolfe steps and steepest descent steps, as well as a novel modification of the ``Frank-Wolfe gap'' to measure convergence in the non-convex case. We further extend this method to incorporate in-face directions for preserving structured solutions as well as block coordinate steps, and we demonstrate computational guarantees in terms of the modified Frank-Wolfe gap for all of these variants. We are particularly motivated by the application of this methodology to the training of neural networks with sparse properties, and we apply our block coordinate method to the problem of $\ell_1$ regularized neural network training. We present the results of several numerical experiments on both artificial and real datasets demonstrating significant improvements of our method in training sparse neural networks.
\end{abstract}

\section{Introduction}\label{sec:intro}
The Frank-Wolfe method (also called the conditional gradient method) and its extensions are often especially applicable in several areas of machine learning due to their low iteration costs and convenient structural properties. The Frank-Wolfe method has classically been applied and analyzed in the setting of smooth, constrained convex optimization problems; for a partial list of references in this setting, see \cite{frank-wolfe, DR1970, Dunn1978} for older references and see \cite{jaggi2013revisiting, harchaoui, GF-FW, freund2017extended} and the references therein for more recent work. At each iteration, the basic Frank-Wolfe method relies only on a single gradient evaluation and a single call to a linear optimization subroutine, wherein the method computes a minimizer of the linear approximation of the objective function over the feasible region and then updates the next iterate as a convex combination of this minimizer and the current iterate.

In this paper, we consider variants of the Frank-Wolfe method for non-convex stochastic optimization problems with mixed constrained and unconstrained variables. Our problem of interest is:
\begin{equation}\label{poi1}
\begin{array}{rcl}
F^* := & \min\limits_{x,y} & F(x, y) := \mathbb{E}_{z\sim \calD}[f(x,y,z)] \\
& \mathrm{s.t.} & x \in S, y \in \bbR^{q} \ ,
\end{array}
\end{equation}
where $S \subseteq \bbR^p$ is a compact and convex set, $z$ is a random variable in a probability space $\calZ$ with (possibly unknown) distribution $\cal D$, and $f(\cdot, \cdot, \cdot) : S \times \bbR^q \times \calZ \to \bbR$ is differentiable in $(x,y)$ for each fixed $z \in \calZ$. Note that we allow for the possibility of either $p = 0$ or $q = 0$, in which case only one of the two sets of variables $x$ and $y$ would be present in \eqref{poi1}. Herein, we develop and analyze several stochastic gradient algorithms that utilize Frank-Wolfe ``style'' steps in the $x$ variables and standard steepest descent steps in the $y$ variables. 

In many core machine learning methodologies, such as the setting of training neural networks, non-convexity is ubiquitous. Moreover, due to large training set sizes, stochastic algorithms (or related strategies) are also a necessity.
For several reasons, including increased interpretability, memory efficiency, and improved computation at prediction/inference time, \emph{structured models} (such as sparse networks, low-rank models, etc.) are often highly desirable.
In order to induce a structured model, one might consider a strategy such as pruning \cite{aghasi2017net, molchanov2016pruning, han2015learning} that modifies the model after the training procedure or one might consider a strategy that induces structured models throughout the training procedure. In this paper, we consider a method that falls into the latter approach based on extending the Frank-Wolfe method to problem \eqref{poi1}. The Frank-Wolfe method, which falls into the more general class of ``structure-enhancing'' algorithms, is particularly attractive because the dynamics of the algorithm directly helps to promote near-optimal well-structured (e.g, sparse, low-rank) solutions. In some optimization formulations, such well-structured solutions also lie on low-dimensional faces of the feasible region, which was a key motivation for the development of ``in-face'' directions (also referred to as alternative directions herein), including away steps \cite{gm, lacoste2015global}, and the in-face extended Frank-Wolfe method developed in \cite{freund2017extended} for the case of deterministic, smooth convex optimization and particularly matrix completion.


In this paper, we extend the methodology of the Frank-Wolfe method with in-face directions to the setting of stochastic non-convex optimization, and we also allow for the possibility of mixed structured and unstructured variables. In other words, the $x$ variables in \eqref{poi1} represent the variables that we would like to be well-structured (e.g., sparse edges in a neural network) and the $y$ variables are completely ``free.'' In Section \ref{sec:algorithm}, we develop a ``hybrid'' Frank-Wolfe steepest descent (FW-SD) method with alternative in-face direction steps that promote structured solutions in the $x$ variables. Although we refer to this as a single method, we prove computational guarantees for two versions:  the simple version without alternative direction steps, and the version with alternative direction steps. 
In the non-convex setting, the ``Frank-Wolfe gap'' function is often used to measure convergence (see, e.g., \cite{lacoste2016convergence, reddi2016stochastic}). 
We introduce a novel modification of the Frank-Wolfe gap that accounts for the mixed constrained and unconstrained variable structure, and all of our theoretical computational guarantees are stated in terms of the modified Frank-Wolfe gap. In particular, if $K$ denotes the total number of iterations and when the number of samples per iteration is $O(K)$, we demonstrate $O(1/K)$ convergence in terms of the expected squared modified Frank-Wolfe gap (i.e., its second moment) for the methods developed herein.
In Section \ref{sec:block}, we extend our method to problems with block coordinate structure. The computational guarantees for the block version is in the worst-case the same as the guarantee for the non-block method, however in practice the block method is effective due to its ability to use different step-sizes for each block.
Section \ref{sec:numerics} presents the results of some numerical experiments on the MNIST and CIFAR-10 datasets demonstrating the viability of our algorithm.




Stochastic gradient methods (also stochastic approximation) dates back to \cite{robbins1951}. For recent works related to stochastic gradient descent and its variants see, e.g., \cite{nemirovski2009robust}, \cite{bottou2010large}, \cite{lan2012optimal}, \cite{bottou2018optimization}, and the references therein. Mostly closely related to our work, at least in terms of the theoretical computational guarantees developed herein, is perhaps \cite{reddi2016stochastic} who study stochastic Frank-Wolfe methods with and without variance reduction in the non-convex setting. In Section \ref{sec:algorithm} we comment on how our results relate to \cite{reddi2016stochastic}.
\cite{ghadimi2016conditional} also studies Frank-Wolfe type methods for stochastic non-convex problems with a composite structure.
\cite{he2015semi} also studies, in the case of convex and related variational inequality problems, Frank-Wolfe type methods with a related ``semi-proximal'' decomposable structure.
Block coordinate Frank-Wolfe type methods have been studied in several contexts beginning with \cite{lacoste2012block}.
Some other related references examining variants and extensions of Frank-Wolfe method in the deterministic setting are \cite{lacoste2016convergence, jiang2016structured, khamaru2018convergence, nouiehed2018convergence, rao2015forward, lacoste2015global, braun2018blended}, and in the stochastic convex setting are \cite{hazan2016variance, goldfarb2017linear, lu2018generalized}.

{\bf Illustrative application: sparse neural network training.}\label{SparseNN}
Let us conclude the introduction by describing an illustrative application to sparse neural network training.
In general, it has been observed that structured neural networks are practically advantageous for several reasons. Networks with desirable structural properties, that are often tailored to the application in mind, are more efficient -- both computationally and statistically -- than general purpose feedforward networks. A general approach for conceptualizing and training well-structured networks is through edge weight sparsity. Indeed, a feedforward network with sparse edges offers a number of benefits including increased interpretability, reduced memory footprint, and reduced computation at prediction/inference time. In fact, in practical applications, inference time and memory footprint are sometimes major bottlenecks that are often overlooked during the training/designing of deep neural networks \cite{canziani2016analysis}. Finally, sparse networks offer a conceptual advantage in that they encompass several popular representations, such as the widely popular convolutional layers. For more discussion on the benefits of sparse networks and approaches for constructing them, see \cite{thom2013sparse, louizos2017learning}, for example.

The formulation for training a sparse neural network considered herein is based on $\ell_1$ regularization, which is a natural and widely popular idea for promoting sparsity as well as other benefits of regularization \cite{hastie2015statistical}. For simplicity, let us consider training a fully connected feedforward network in the general setting of supervised learning. (This model can easily be extended to an arbitrary directed acylic graph.) We consider a per-node regularization model where including an $\ell_1$ regularization constraint is optional for each node, hence we define the set $\mathcal{N}_R := \{(t,i) : \text{ node } i \text{ in layer } t \text{ imposes a regularization constraint}\}$ and we let $\delta_{t, i}$ denote the corresponding regularization parameter. The optimization model we consider is:
\begin{equation}\label{poi_network}
\min\limits_{w}\displaystyle\frac{1}{n}\sum_{i = 1}^n \ell(\hat{y}_w(x_i), y_i) \ , \ \ \ \text{s.t.} \ \ \|w^{t-1}_i\|_1 \leq \delta_{i, t} \ \text{ for all } (t,i) \in \mathcal{N}_R \ ,
\end{equation}
where $(x_1, y_1), \ldots, (x_n, y_n)$ is the training data (in this section only we use $x$ and $y$ to refer to data, in later sections they refer to the optimization variables), $\hat{y}_w(\cdot) : \bbR^d \to \bbR^l$ denotes the prediction function of the model parameterized by the collection of weights $w$, and $\ell(\cdot, \cdot) : \bbR^l \times \bbR^l \to \bbR$ is a differentiable loss function. Here $w^{t-1}_i$ is the vector of incoming edge weights at node $i$ in layer $t$.

We note that a very closely related model, albeit with less flexibility, has been studied in the improper learning setting by \cite{zhang2016l1}. Let us now point out a few salient features of the optimization model \eqref{poi_network}. First, note that the variables of \eqref{poi_network} can be partitioned into constrained variables and unconstrained variables, corresponding to $x$ and $y$ in \eqref{poi1}, respectively. Furthermore, the constrained variables in \eqref{poi_network} have a block coordinate decomposable structure, wherein each vector $w_{t-1, i}$ for $(t,i) \in \mathcal{N}_R$ is constrained only to lie in its own $\ell_1$ ball of radius $\delta_{i,t}$. This type of block coordinate decomposable structure is considered herein in Section \ref{sec:block}. 
Note that the $\ell_1$ ball constraints are intended to promote sparsity, and therefore algorithmic schemes that also promote sparsity, such as in-face directions, are highly desirable in this context. It is possible to also model additional types of network structures with different types of convex constraints that are amenable to the Frank-Wolfe method and its extensions. For example, node level sparsity can be modeled with group $\ell_1$ constraints and low-rank weight matrices between can be modeled with nuclear norm constraints.

\section{Stochastic Frank-Wolfe steepest descent method with in-face directions}\label{sec:algorithm}
Let us now return to studying the generic non-convex stochastic optimization problem \eqref{poi1} where $x$ is constrained to lie in a compact and convex set $S$ and $y$ is unconstrained.
As mentioned, our algorithm is based on using Frank-Wolfe steps in the $x$ variables and steepest descent steps in the $y$ variables (both with stochastic versions of the partial gradients). Let us first review some useful notation.

{\bf Notation.}
Let $\|\cdot\|_X$ be a given norm on the variables $x \in \mathbb{R}^p$, and let $\|\cdot\|_Y$ be a given norm on the variables $y \in \mathbb{R}^q$. The diameter of $S$ is $\diam(S) := \max_{x, \bar{x} \in S}\|x - \bar{x}\|_X$, and recall that $\diam(S) < +\infty$ since $S$ is bounded.
The dual norms associated with $\|\cdot\|_X$ and $\|\cdot\|_Y$ are denoted by $\|\cdot\|_{X\ast}$ and $\|\cdot\|_{Y\ast}$, respectively. Recall that $\|\cdot\|_{X\ast}$ is defined by $\|s\|_{X\ast} := \max_{x : \|x\| \leq 1}s^Tx$ and $\|\cdot\|_{Y\ast}$ is defined analogously.
We also use $\|\cdot\|$ to denote the ``Euclidean combination'' of the two norms $\|\cdot\|_X$ and $\|\cdot\|_Y$ as the norm on $(x, y) \in \bbR^p \times \bbR^q$, whereby $\|(x,y)\| := \sqrt{\|x\|_X^2 + \|y\|_Y^2}$. Note that the dual norm of $\|\cdot\|$ is also the ``Euclidean combination'' of $\|\cdot\|_{X\ast}$ and $\|\cdot\|_{Y\ast}$, whereby $\|(s,t)\|_\ast = \sqrt{\|s\|_{X\ast}^2 + \|t\|_{Y\ast}^2}$. 
The standard inner product between $s \in \bbR^p$ and $x \in \bbR^p$ is denoted by $s^Tx$, and that the inner product on $\bbR^p \times \bbR^q$ is the sum of the inner products on the two spaces, i.e, $(s,t)^T(x,y) := s^Tx + t^Ty$.
The notation $\nabla$ refers to gradients with respect to $(x,y)$, and $\nabla_x$ and $\nabla_y$ refers to partial gradients with respect to $x$ and $y$, respectively.
For a scalar $\alpha$, $\sgn(\alpha)$ is the sign of $\alpha$, which is equal to $-1$ if $\alpha < 0$, $+1$ if $\alpha > 0$ and $0$ if $\alpha = 0$.
The notation ``$\tilde v \leftarrow \argmax_{v \in S} \{f(v)\}$'' denotes assigning $\tilde v$ to be an arbitrary optimal solution of the problem $\max_{v \in S} \{f(v)\}$.

{\bf Assumptions.} Note that the choice of the norm $\|\cdot\|_{Y}$ directly affects the form of the steepest descent step. For example, if $\|\cdot\|_{Y}$ is the $\ell_2$ norm then the steepest descent step becomes a standard stochastic gradient step. Another relevant example is when $\|\cdot\|_{Y}$ is the $\ell_1$ norm, in which case the steepest descent step becomes a stochastic variant of a greedy coordinate descent step (see, e.g., \cite{nutini2015coordinate}). 
On the other hand, the choice of the norm $\|\cdot\|_X$ does not affect the direction of the Frank-Wolfe step in the $x$ variables but it does affect the step-size strategy employed herein.

We make the following assumptions regarding problem \eqref{poi1}:
\begin{enumerate}
\item[(A1)] The objective function $F(\cdot, \cdot)$ is smooth, i.e., there is a constant $L_{\nabla} > 0$ such that $\|\nabla F(x, y) - \nabla F(\bar{x}, \bar{y})\|_\ast \leq L_{\nabla}\|(x, y) - (\bar{x}, \bar{y})\|$ for all $x, \bar{x} \in S$ and $y, \bar{y} \in \bbR^q$.
\item[(A2)] The partial gradient with respect to $x$ is uniformly bounded, i.e., there is a constant $L_f > 0$ such that $\|\nabla_x f(x, y, z)\|_{X\ast} \leq L_f$ for all $x \in S, y \in \bbR^q$, and $z \in \calZ$.
\item[(A3)] The stochastic gradient has bounded variance, i.e., there is a constant $\sigma \geq 0$ such that $\bbE_{z \sim D}\left[\|\nabla f(x, y, z) - \nabla F(x, y)\|_\ast^2\right] \leq \sigma^2$ for all $x \in S$ and $y \in \bbR^q$.
\item[(A4)] We have knowledge of the constant $L_\nabla$ as well as a constant $\bar{C} > 0$ satisfying $\bar{C} \geq \max\left\{2L_{\nabla}\cdot\text{diam}(S)^2, L_f\cdot\text{diam}(S)\right\}$.
\end{enumerate}

{\bf Modified Frank-Wolfe gap.}
Since \eqref{poi1} is generally a non-convex problem, we measure convergence in terms of a 
modified version of the ``Frank-Wolfe'' gap function. Let us first define the function $\tilde G(\cdot, \cdot) : S \times \bbR^q \to \bbR_+$ by $\tilde G(\bar x, \bar y) := \max_{x \in S}\left\{\nabla_x F(\bar x, \bar y)^T(\bar x - x)\right\}$. Note that when the $y$ variables are not present, this is exactly the definition of the ``gap function'' due to \cite{hearn1982gap} and studied in the recent literature on Frank-Wolfe.
Definition \ref{def:gap} presents our modified gap function that accounts for both the $x$ and $y$ variables.
\begin{definition}\label{def:gap}
The modified Frank-Wolfe gap function $G(\cdot, \cdot) : S \times \bbR^q \to \bbR_+$ is the function given by
\begin{equation*}
G(\bar x, \bar y) := \tilde G(\bar x, \bar y)\sqrt{\frac{2L_\nabla}{\bar{C}}} ~+~ \|\nabla_y F(\bar x, \bar y)\|_{Y\ast} \ ,
\end{equation*}
where $\tilde G(\bar x, \bar y) := \max_{x \in S}\left\{\nabla_x F(\bar x, \bar y)^T(\bar x - x)\right\}$.
\end{definition}
Note that Definition \ref{def:gap} depends on the particular specification of the parameters $L_\nabla$ and $\bar{C}$, which is a slightly undesirable property. However, in the case when $\bar{C} = 2L_{\nabla}\cdot\text{diam}(S)^2$, then we have that $\sqrt{\frac{2L_\nabla}{\bar{C}}} = \frac{1}{\diam(S)}$, which is a natural way to normalize the function $\tilde G(\cdot, \cdot)$. Note again that when the $y$ variables are not present then the modified Frank-Wolfe gap reduces to a scaled version of the standard Frank-Wolfe gap, and when the $x$ variables are not present then it reduces to the norm of the gradient (which is also a standard metric in unconstrained non-convex optimization).
The use of the modified Frank-Wolfe gap is justified by Proposition \ref{prop:gap} below, which states that $G(\bar{x}, \bar{y}) = 0$ is a \emph{necessary condition} for any locally optimal solution $(\bar{x}, \bar{y})$.
\begin{proposition}\label{prop:gap}
Suppose that $(\bar{x}, \bar{y})$ is a locally optimal solution of problem \eqref{poi1}. Then, it holds that $G(\bar{x}, \bar{y}) = 0$.
\end{proposition}
The proof of Proposition \ref{prop:gap}, as well as all other omitted proofs, is included in the supplementary materials.
In the convex case, we can also use the modified Frank-Wolfe gap to bound the objective function value optimality gap, as demonstrated by Proposition \ref{prop:gap_convex} below.
\begin{proposition}\label{prop:gap_convex}
Suppose that $F(\cdot, \cdot)$ is convex on $S \times \bbR^q$, and let $(x^\ast, y^\ast)$ denote an optimal solution of \eqref{poi1}. Consider a given feasible solution $(\bar{x}, \bar{y}) \in S \times \bbR^q$, and let $R \geq 0$ be a constant such that $\|\bar{y} - y^\ast\|_Y \leq R$. Then, it holds that:
\begin{equation*}
F(\bar{x}, \bar{y}) - F^\ast ~\leq~ \max\left\{\sqrt{\tfrac{\bar{C}}{2L_\nabla}}, R\right\} \cdot G(\bar{x}, \bar{y}) \ .
\end{equation*}
\end{proposition}
Note that Proposition \ref{prop:gap_convex} requires existence of a constant $R \geq 0$ such that $\|\bar{y} - y^\ast\|_Y \leq R$. In the case that $(\bar{x}, \bar{y})$ corresponds to the iterate of some algorithm that is guaranteed to lie in a bounded initial level set of the function $F(\cdot, \cdot)$, then this constant $R$ is guaranteed to exist. For example, this is always the case for deterministic steepest descent. This level set condition is not guaranteed to hold for the stochastic algorithms that we study herein but we would expect this condition to hold in practice (with high probability) after sufficiently many iterations.
We also utilize Lemma \ref{lem:jensen_gap} below, which relates a stochastic estimate of the modified gap function to the above definition.
\begin{lemma}\label{lem:jensen_gap}
Let $(\bar{x}, \bar{y}) \in S \times \bbR^q$ be given and let $(\hat{g}, \hat{h})$ denote an unbiased stochastic estimate of $\nabla F(\bar{x}, \bar{y})$ such that $\bbE[\hat{g}] = \nabla_x F(\bar x, \bar{y})$ and $\bbE[\hat{h}] = \nabla_y F(\bar x, \bar{y})$. Define the random variables:
\begin{equation*}
\tilde G := \max_{x \in S}\left\{\hat{g}^T(\bar{x} - x)\right\} \ , \text{ and } \ \hat G := \tilde{G}\sqrt{\frac{2L_\nabla}{\bar{C}}} ~+~ \|\hat{h}\|_{Y\ast} \ .
\end{equation*}
Then, it holds that $\bbE[\hat{G}] \geq G(\bar x, \bar y)$.
\end{lemma}

\subsection{Stochastic Frank-Wolfe steepest descent (FW-SD) method with in-face directions}
We are now ready to present our stochastic Frank-Wolfe steepest descent method for problem \eqref{poi1}, which possibly incorporates ``alternative directions'' and is formally presented below in Algorithm \ref{alg:stochastic_alt_fw_sd}. Algorithm \ref{alg:stochastic_alt_fw_sd} includes a true/false variable, called $\mathrm{AlternativeDirections}$, which indicates whether to use the alternative direction step in the $x$ variables or not.
Each iteration of Algorithm first uses a stochastic estimate of the gradient $\nabla F(x_k, y_k)$ based on $b_k$ i.i.d. samples to perform standard Frank-Wolfe and steepest descent steps in the variables $x$ and $y$, respectively.
Note that the step-sizes $\bar{\alpha}_k$ and $\alpha_k$ are dynamic random variables depending on the stochastic gradients $\hat g_k$ and $\hat h_k$, which is in contrast to the step-sizes (such as constant step-sizes in \cite{reddi2016stochastic}) that have been previously considered in the literature on stochastic Frank-Wolfe methods in the non-convex setting. The dynamic step-sizes employed by Algorithm \ref{alg:stochastic_alt_fw_sd} are more ``adaptive'' than constant step-sizes and hence can have better practical performance.
If alternative directions are not used, then $x_{k+1}$ is simply determined by the stochastic Frank-Wolfe step. Otherwise, for the version with alternative directions, Step (4.) is the computation of the stochastic alternative direction step, which is based on a fresh stochastic gradient estimate in Step (4a.). Then, Step (4b.) represents the computation of a generic alternative direction $d_k$ and the corresponding step, which we elaborate on further below.
Finally, note that the output of Algorithm \ref{alg:stochastic_alt_fw_sd} is chosen uniformly at random from all past iterates, which is also equivalent to randomly sampling the total number of iterations prior to starting the algorithm.

\floatname{algorithm}{Algorithm}
\begin{algorithm}
\caption{Stochastic Frank-Wolfe steepest descent (FW-SD) Method with alternative directions} \label{alg:stochastic_alt_fw_sd}
\begin{algorithmic}
\STATE {\bf Initialize} at $x_0 \in S$, $y_0 \in \bbR^q$ $k \gets 0$, set $\mathrm{AlternativeDirections} \in \{\text{TRUE}, \text{FALSE}\}$. \\
\medskip

{\bf At iteration $k$:}
\STATE 1. Choose number of samples $b_k$, sample $z_{k,1}, \ldots, z_{k,b_k}$ i.i.d. from $\calD$ and compute:
\begin{description}
\item $\ \  \hat{g}_k \gets \frac{1}{b_k} \sum_{i = 1}^{b_k} \nabla_x f(x_k, y_k, z_{k,i})$
\item $\ \  \hat{h}_k \gets \frac{1}{b_k} \sum_{i = 1}^{b_k} \nabla_y f(x_k, y_k, z_{k,i})$
\end{description}\medskip
\STATE 2. Do Stochastic Frank-Wolfe Step:
\begin{description}
\item $\ \  \tilde x_k \gets \arg\min\limits_{x \in S}\{\hat{g}_k^Tx\}$
\item $\ \  \tilde G_k \gets \hat{g}_k^T(x_k - \tilde{x}_k)$
\item $\ \  \bar{x}_{k} \gets x_k + \bar{\alpha}_k(\tilde{x}_k - x_k)$ where $\bar\alpha_k := \tilde{G}_k/\bar{C}$
\item \ \ {\bf If $\mathbf{AlternativeDirections} = $ FALSE, then set} $x_{k+1} \gets \bar{x}_{k}$
\end{description}\medskip
\STATE 3. Do Stochastic steepest descent Step:
\begin{description}
\item $\ \ \tilde{y}_k \gets \arg\max\limits_{y \in \bbR^q}\{\hat{h}_k^T y : \|y\|_Y \le 1 \}$
\item $\ \  y_{k+1} \gets y_k - \alpha_k\tilde{y}_k$ where $\alpha_k := \|\hat{h}_k\|_{Y\ast}/2L_{\nabla}$
\end{description}\medskip
\STATE 4. {\bf If $\mathbf{AlternativeDirections} = $ TRUE, then do} Stochastic Alternative Direction Step:
\STATE \ \ \ \ 4a. Sample $\check{z}_{k,1}, \ldots, \check{z}_{k,b_k}$ i.i.d. from $\calD$ and compute 
\begin{description}
\item $\ \ \ \ \ \ \ \ \check{g}_k \gets \frac{1}{b_k} \sum_{i = 1}^{b_k} \nabla_x f(\bar{x}_k, y_{k+1}, \check{z}_{k,i})$
\end{description}\medskip
\STATE \ \ \ \ 4b. Compute a stochastic alternative direction $d_k$ (formally a measurable function of $\check{g}_k$)
\STATE \ \ \ \ \ \ \ \ \ \ satisfying $\check{g}_k^Td_k < 0$ and $\|d_{k}\|_{X} \leq \diam(S)$, and set:
\begin{description}
\item $\ \ \ \ \ \ \ \ A_k := - \check{g}_k^Td_k$
\item $\ \ \ \ \ \ \ \ \alpha^{\text{stop}}_k := \arg\max\limits_{\alpha \geq 0}\{\alpha : \bar{x}_k + \alpha d_k \in S\}$
\item $\ \ \ \ \ \ \ \ x_{k+1} \gets \bar{x}_k + \bar{\beta}_k d_k$ where $\bar{\beta}_k := \min\left\{A_k/\bar{C}, \alpha^{\text{stop}}_k\right\}$
\end{description}\medskip

{\bf After $K$ total iterations:}
\STATE {\bf Output:} $(\hat x_k, \hat y_k)$ chosen uniformly at random from $(x_0, y_0), \ldots, (x_{K}, y_K)$
\end{algorithmic}
\end{algorithm}




Theorem \ref{fw-sd-theorem} below presents our main computational guarantee for the stochastic Frank-Wolfe steepest descent method for the non-convex optimization problem \eqref{poi1}.
The statement of the theorem involves the maximum ratios between the Euclidean norm $\|\cdot\|_2$ and the given norm $\|\cdot\|$, defined by:
\begin{equation*}
\kappa_1 := \max_{(x,y) \ne 0} \|(x,y)\|_2/\|(x,y)\| \ , \ \ 
\kappa_2 := \max_{(x,y) \ne 0} \|(x,y)\|/\|(x,y)\|_2 \ .
\end{equation*}
(Note that $\|(x,y)\|_2$ is simply defined by $\|(x,y)\|_2 := \sqrt{\|x\|_2^2 + \|y\|_2^2}$.) Let us also define $\kappa := \kappa_1\kappa_2$. Note that norm equivalence on finite dimensional vector spaces ensures that $\kappa$ is finite, and in the case that $\|\cdot\|_X$ and $\|\cdot\|_Y$ are both the $\ell_2$ norm then $\kappa = 1$.
We also define a constant $\alpha_{\mathrm{AD}}$ that is useful in the statement of the theorem as well as later results. Specifically, we define $\alpha_{\mathrm{AD}} := 1$ if $\mathrm{AlternativeDirections} = \text{FALSE}$ and $\alpha_{\mathrm{AD}} := 2$ if $\mathrm{AlternativeDirections} = \text{TRUE}$.

\begin{theorem}\label{fw-sd-theorem}
Consider the Stochastic FW-SD Method, possibly with alternative directions (Algorithm \ref{alg:stochastic_alt_fw_sd}).
Under assumptions (A1)-(A4), it holds for all $K \geq 0$ that:
\begin{equation*}
\bbE[G(\hat{x}_K, \hat{y}_K)^2] ~\leq~ \frac{8L_\nabla(F(x_0, y_0) - F^\ast)}{K+1} ~+~ \frac{4\alpha_{\mathrm{AD}}\kappa^2\sigma^2}{K+1}\sum_{k = 0}^K\frac{1}{b_k} \ .
\end{equation*}
\end{theorem}
Based on Theorem \ref{fw-sd-theorem}, setting $b_k = K$ at each iteration of Algorithm \ref{alg:stochastic_alt_fw_sd} leads to an $O(1/K)$ convergence bound whereas setting $b_k = k$ would lead to an $O(\ln(K)/K)$ bound.
Note that, as compared to previous related results for the Frank-Wolfe method in the non-convex case developed in \cite{reddi2016stochastic}, Theorem \ref{fw-sd-theorem} obtains a similar bound but has a few differences. In addition to the novel extensions herein of including steepest descent steps in the $y$ variables and possibly incorporating alternative direction steps in $x$ variables, note that Theorem \ref{fw-sd-theorem} holds for the dynamic step-size rule of Algorithm \ref{alg:stochastic_alt_fw_sd} whereas \cite{reddi2016stochastic} studies a constant step-size rule. Also note that Theorem \ref{fw-sd-theorem} bounds the second moment of the modified Frank-Wolfe gap whereas \cite{reddi2016stochastic} bounds the first moment of the Frank-Wolfe gap.
The proof of Theorem \ref{fw-sd-theorem} is included in the supplement.



{\bf Examples of alternative ``in-face'' directions.}
Step (4b.) of Algorithm \ref{alg:stochastic_alt_fw_sd} is written in a purposefully generic way that does not specify precisely how to compute the alternative direction $d_k$ so that we may consider a wide framework that accommodates several computationally advantageous choices. 
At the same time, the intuitive idea of the role of alternative directions in the convergence analysis of Algorithm \ref{alg:stochastic_alt_fw_sd} is that alternative directions should do no harm in terms of the modified Frank-Wolfe gap convergence. Let us now present several concrete examples of alternative directions, which all have the property of also being in-face directions. Given any feasible $x \in S$, we denote ${\cal F}_S(x)$ as the minimal face of $S$ that contains the point $x$.
Given $\bar{x}_k$ computed in Step (2.) of Algorithm \ref{alg:stochastic_alt_fw_sd}, $d_k$ is an in-face direction if $\bar{x}_k + \alpha d_k \in {\cal F}_S(\bar{x}_k)$ for all $\alpha \in [0, \alpha^{\text{stop}}_k]$.
One possible in-face direction is the ``away step'' direction introduced in \cite{gm} obtained by choosing $d_k \gets x_k - \check x_k \ , \ \ \ \text{where} \ \ \ \check x_k \gets  \arg\max_{x \in {\cal F}_S(\bar{x}_k)}\{\check{g}_k^Tx\}$.
Due to the facial structure of $S$, in-face directions often preserve certain types of solution structures. For example, when $S$ is an $\ell_1$ ball, then an in-face step preserves sparsity so that $x_{k+1}$ has the the same signed sparsity pattern as $\bar{x}_k$. In-face directions, including away steps, are also often as simple to compute or even simpler to compute than the Frank-Wolfe directions. Another example of an in-face directions pertinent to the non-convex setting include a regular Frank-Wolfe direction inside ${\cal F}_S(\bar{x}_k)$.
Additional discussion of how to compute in-face directions in the case where $S$ is an $\ell_1$ ball is included in the supplementary materials.

\section{Block coordinate extension}\label{sec:block}
In this section, we extend the previously developed stochastic FW-SD method with alternative directions to the block coordinate setting. We consider an extension of problem \eqref{poi1} where the variables $x$ as well as the feasible region $S$ for $x$ have a block coordinate structure. Specifically, we presume that $x \in S \subseteq \bbR^p$ has a decomposable block coordinate structure across $N \geq 1$ total blocks, whereby $x = (x^{(1)}, \ldots, x^{(N)})$, $S = S_1 \times \cdots \times S_N$, and each $x^{(i)} \in S_i \subseteq \bbR^{p_i}$ where $S_i$ is a compact and convex set and with $\sum_{i = 1}^N p_i = p$.
For each $i \in \{1, \ldots, N\}$, let $\|\cdot\|_{X,i}$ denote the given norm on the space of variables $x^{(i)} \in \bbR^{p_i}$ with dual norm denoted by $\|\cdot\|_{X\ast,i}$. The norm $\|\cdot\|_{X}$ on the overall space of $x = (x^{(1)}, \ldots, x^{(N)})$ variables is now taken to be the Euclidean combination of all of the block norms, i.e., we define $\|x\|_X := \sqrt{\sum_{i = 1}^N \|x^{(i)}\|_{X,i}^2}$. Note that the overall norm on the entire space of variables $(x,y)$ is the same as before, i.e., $\|(x,y)\| := \sqrt{\|x\|_X^2 + \|y\|_Y^2}$.
Furthermore, recall that $\diam(S_i) := \max_{x^{(i)}, \bar{x}^{(i)} \in S}\|x^{(i)} - \bar{x}^{(i)}\|_{X,i}$. We use the notation $\nabla_x^{(i)}$ to refer to partial gradients with respect to $x^{(i)}$ for each $i \in \{1, \ldots, N\}$.

In this block coordinate setting, we retain the earlier assumptions (A1) and (A3) and modify assumptions (A2) and (A4) as follows:
\begin{enumerate}
\item[(A2-B)] For each block $i \in \{1, \ldots, N\}$, the partial gradient with respect to $x^{(i)}$ is uniformly bounded, i.e., there is a constant $L_{f,i} > 0$ such that $\|\nabla_x^{(i)} f(x, y, z)\|_{X\ast,i} \leq L_{f,i}$ for all $x \in S, y \in \bbR^q$, and $z \in \calZ$.
\item[(A4-B)] We have knowledge of the constant $L_\nabla$ as well as constants $\bar{C}_i > 0$ satisfying $\bar{C}_i \geq \max\left\{2L_{\nabla}\cdot\text{diam}(S_i)^2, L_{f,i}\cdot\text{diam}(S_i)\right\}$ for each block $i \in \{1, \ldots, N\}$.
\end{enumerate}

Notice that the decomposable structure of $S$, i.e., $S = S_1 \times \cdots \times S_N$ implies that linear optimization problems are completely separable across the $N$ different blocks and that the modified Frank-Wolfe gap also has a similar decomposable structure.
For each block $i \in \{1, \ldots, N\}$, let us define $\tilde G_i(\cdot, \cdot) : S \times \bbR^q \to \bbR_+$ by $\tilde G_i(\bar x, \bar y) := \max_{x^{(i)} \in S_i}\left\{\nabla_x^{(i)} F(\bar x, \bar y)^T(\bar{x}^{(i)} - x^{(i)})\right\}$. Then, the function $\tilde{G}(\cdot, \cdot)$ defined in Section \ref{sec:algorithm} satisfies $\tilde{G}(\bar{x}, \bar{y}) = \sum_{i = 1}^N \tilde G_i(\bar x, \bar y)$ for all $x \in S$ and $y \in \bbR^q$. Moreover, in light of Assumption (A4B), we have that $\sum_{i = 1}^N \bar{C}_i \geq 2L_{\nabla}\sum_{i = 1}^N \diam(S)_i^2 = 2L_{\nabla}\diam(S)^2$. Hence, we define $\bar{C} := \sum_{i = 1}^N \bar{C}_i$, which is needed to specify the modified Frank-Wolfe gap function.

The main idea of the block coordinate version of Algorithm \ref{alg:stochastic_alt_fw_sd} is to replace the stochastic Frank-Wolfe step in Step (2.) and the alternative direction step in Step (4.) with block coordinate versions that use different step-sizes in each of the different blocks. Subroutines \ref{alg:block_fw} and \ref{alg:block_alt} below precisely describe how the block variants of these two steps work. 

\floatname{algorithm}{Subroutine}
\begin{algorithm}
\caption{Block Coordinate Stochastic Frank-Wolfe Step} \label{alg:block_fw}
\begin{algorithmic}
\STATE For each $i = 1, \ldots, N$, set:
\begin{description}
\item $\ \  \tilde x_k^{(i)} \gets \arg\min\limits_{x^{(i)} \in S_i}\{(\hat{g}_k^{(i)})^Tx^{(i)}\}$
\item $\ \  \tilde G_k^{i} \gets (\tilde{g}_k^{(i)})^T(x_k^{(i)} - \tilde{x}_k^{(i)})$
\item $\ \  \bar{x}^{(i)}_{k} \gets x_k^{(i)} + \bar{\alpha}_k^{i}(\tilde{x}_k^{(i)} - x_k^{(i)})$ where $\bar\alpha_k^i := \tilde{G}_k^{i}/\bar{C}_i$. 
\end{description}
\end{algorithmic}
\end{algorithm}

\floatname{algorithm}{Subroutine}
\begin{algorithm}
\caption{Block Coordinate Stochastic Alternative Direction Step} \label{alg:block_alt}
\begin{algorithmic}
\STATE 4a. Sample $\check{z}_{k,1}, \ldots, \check{z}_{k,b_k}$ i.i.d. from $\calD$ and compute 
\begin{description}
\item $\ \  \check{g}_k \gets \frac{1}{b_k} \sum_{i = 1}^{b_k} \nabla_x f(\bar{x}_k, y_{k+1}, \check{z}_{k,i})$
\end{description}\medskip
\STATE 4b. Compute a stochastic alternative direction $d_k$ (formally a measurable function of $\check{g}_k$) satisfying $(\check{g}_k^{(i)})^Td_k^{(i)} < 0$ and $\|d_{k}^{(i)}\|_{X,i} \leq \diam(S_i)$ for all $i \in \{1, \ldots, N\}$. For each $i \in \{1, \ldots, N\}$, set:
\begin{description}
\item $\ \ A_k^i := - (\check{g}_k^{(i)})^Td_k^{(i)}$
\item $\ \ \alpha^{\text{stop}, i}_k := \arg\max\limits_{\alpha \geq 0}\{\alpha : \bar{x}_k^{(i)} + \alpha d_k^{(i)} \in S_i\}$
\item $\ \  x_{k+1}^{(i)} \gets \bar{x}^{(i)}_k + \bar{\beta}_k^i d_k^{(i)}$ , $\bar{\beta}_k^i := \min\left\{A_k^i/\bar{C}_i, \alpha^{\text{stop}, i}_k\right\}$
\end{description}
\end{algorithmic}
\end{algorithm}


\begin{theorem}\label{block-alt-fw-sd-theorem}
Consider the Block Coordinate Stochastic FW-SD Method, possibly with alternative directions, i.e., Algorithm \ref{alg:stochastic_alt_fw_sd} with Step (2.) replaced with Subroutine \ref{alg:block_fw} and Step (4.) replaced with Subroutine \ref{alg:block_alt}. 
Under assumptions (A1), (A2B), (A3), and (A4B), it holds for all $K \geq 0$ that:
\begin{equation*}
\bbE[G(\hat{x}_K, \hat{y}_K)^2] ~\leq~ \frac{8L_\nabla(F(x_0, y_0) - F^\ast)}{K+1} ~+~ \frac{4\alpha_{\mathrm{AD}}\kappa^2\sigma^2}{K+1}\sum_{k = 0}^K\frac{1}{b_k} \ ,
\end{equation*}
where the modified Frank-Wolfe gap $G(\cdot, \cdot)$ (Definition \ref{def:gap}) is defined using $\bar{C} := \sum_{i = 1}^N \bar{C}_i$.
\end{theorem}


\section{Numerical Experiments}\label{sec:numerics}
Let us now discuss our illustrative numerical experiments wherein we applied the block coordinate version Algorithm \ref{alg:stochastic_alt_fw_sd} studied in Section \ref{sec:block} to the $\ell_1$ regularized neural network training problem \eqref{poi_network} on both synthetic and real datasets.
We used PyTorch \cite{paszke2017automatic} to write an optimizer that partitions the layers into:  {\em (i)} Frank-Wolfe layers whose weights correspond to the $x$ variables in \eqref{poi1}, and {\em (ii)} SGD layers whose weights correspond to the $y$ variables in \eqref{poi1} and with the $\ell_2$ norm used for the steepest descent steps.
For the type of in-face direction in the Frank-Wolfe layers, we used away steps on the $\ell_1$ ball as described earlier and elaborated on further in the supplementary materials.
We initialize the weights of the Frank-Wolfe layers in such a way that each node has at least one non-zero edge coming in and another coming out. Since the Lipschitz constant may not be known in practice, we used cross validation on a held out validation set to tune the parameter $L_{\nabla}$ over the range $L_\nabla =4^{i}$ with $i \in \{-1,0,\ldots,6\}$.
Finally, since our method is not much more complex than SGD and is supported by rigorous computational guarantees, our experiments are intended to be illustrative.
In particular, we would like to illustrate the potential advantages of incorporating Frank-Wolfe layers on top of layers that use standard SGD or SGD variants (e.g., momentum). Therefore, we only perform comparisons with the basic SGD method which uses the standard PyTorch initialization.
We also try both variants of the block coordinate version of Algorithm \ref{alg:stochastic_alt_fw_sd}, referred to as SFW (Stochastic Frank-Wolfe Steepest Descent without alternative in-face directions) and SFW-IF (Stochastic Frank-Wolfe Steepest Descent with alternative in-face directions) herein. (Note that, out of fairness with respect to the number of stochastic gradient calls, we allow SFW to have twice as many iterations as SFW-IF by counting each iteration of Algorithm \ref{alg:stochastic_alt_fw_sd} as two iterations in the case when $\mathrm{AlternativeDirections} = $ TRUE.)
Finally, note that all methods were run for 25 epochs using a batch size of 250 data points.

We experimented with a multilayer perceptron and a convolutional network for MNIST and a convolutional network for CIFAR-10. The convolutional networks for MNIST and CIFAR-10 were taken from PyTorch tutorials \cite{MNISTConv, CIFARConv}, while the multilayer perceptron is simply a three layer network taken from a Keras tutorial \cite{MNISTMLP}. For the multilayer perceptron MNIST example, we treat the first two layers as Frank-Wolfe layers. 
For the convolutional CIFAR-10 and MNIST examples, we treated the convolutional layers as SGD layers and the next two dense layers after the convolutional layers as Frank-Wolfe layers.
The bias terms are always incorporated into the SGD variables.
For SFW and SFW-IF, we cross validated $\delta$ separately for each layer on a grid of values $\{1, 5, 10, 50, 100\}$. Each of the examples has two Frank-Wolfe layers, and for each of these layers we report the average percent of non-zero edges going into each node (we consider any value less than 0.001 to be 0) of the solutions returned after 25 epochs by the three methods. For the same solutions returned by the three methods, we examined how the test accuracy is affected when we do hard thresholding to retain only the top $\theta\%$ of largest magnitude edges in the Frank-Wolfe layers. The results are displayed in Table \ref{tab:RealExps}, which shows that SFW-IF and SFW outperform SGD in these two metrics. All experiments show that SFW and SFW-IF are more robust to hard-thresholding than SGD. For example, when we zero out 95\% of the entries in MNIST-MLP the solution found by SFW-IF only losses 0.4\% of accuracy, while SGD loses more than 17\%. Interestingly, on the convolutional networks, SFW can be more robust than SFW-IF for very small values of the hard-thresholding parameter $\theta$.
Also, the results show that both SFW and SFW-IF find solutions in which most of the weight entries have very small values ($<0.001$), while SGD simply does not promote this behaviour. We also performed experiments on synthetically generated data which are described in detail in the supplementary materials.

\begin{table}[h!]\label{tab:RealExps}
\centering
\resizebox{\textwidth}{!}{%
\begin{tabular}{@{}ccccccccccccc@{}}
\toprule
\multicolumn{12}{c}{\bf{MNIST and CIFAR-10 Results} \medskip} \\
 & & \multicolumn{3}{c}{{\bf MNIST-MLP}} & & \multicolumn{3}{c}{{\bf MNIST-Conv}} & & \multicolumn{3}{c}{{\bf CIFAR-10}} \\
 \cline{3-5} \cline{7-9} \cline{11-13} \\
\bf{Metric} &  & \bf{SFW-IF} & \bf{SFW} & \bf{SGD} & & \bf{SFW-IF} & \bf{SFW} & \bf{SGD} & & \bf{SFW-IF} & \bf{SFW} & \bf{SGD} \\ \midrule 
Layer 1 Avg. NNZ (\%) &  & 10.05 & 7.26 & 97.28 & & 9.71 & 1.69 & 97.23 & & 29.27 & 15.12 & 98.08 \\
Layer 2 Avg. NNZ (\%) &  & 1.55 & 0.73 & 97.79 & & 27.34 & 13.08 & 98.78 & & 7.27 & 13.26 & 98.90 \\ \\ 
Accuracy (\%) w/ Top 100\% &  & 96.88 & 96.49 & 98.25 & & 98.79 & 98.50 & 99.11 & & 54.72 & 53.12 & 57.76 \\
Accuracy (\%) w/ Top 50\% &  & 96.88 & 96.49 & 98.12 & & 98.79 & 98.50 & 99.07 & & 54.71 & 53.12 & 55.69 \\
Accuracy (\%) w/ Top 25\% &  & 96.88 & 96.46 & 97.73 & & 98.82 & 98.50 & 98.63 & & 54.33 & 53.12 & 49.04 \\
Accuracy (\%) w/ Top 10\% &  & 96.73 & 96.22 & 96.80 & & 98.58 & 98.49 & 96.44 & & 44.82 & 52.57 & 37.08 \\
Accuracy (\%) w/ Top 5\% &  & 96.49 & 94.52 & 81.12 & & 90.8 & 98.06 & 84.38 & & 32.68 & 49.5 & 23.63 \\ \\  \bottomrule
\end{tabular}
}
\caption{}
\label{small-table}
\end{table}





\clearpage 

\section*{Acknowledgements}
This research is supported by NSF Awards CCF-1755705 and CMMI-1762744.

\bibliographystyle{abbrv}
\bibliography{GF-papers-orc_student_paper}

\clearpage

\appendix

{\Large \bf Supplementary Materials}

\section{Proofs in Section 2}

\subsection{Proof of Proposition 2.1}
\begin{proof}
It is easily verified that $\tilde{G}(\bar{x}, \bar{y}) \geq 0$ and hence $G(\bar{x}, \bar{y}) \geq 0$ for all $(\bar{x}, \bar{y}) \in S \times \bbR^q$. Now suppose that $G(\bar{x}, \bar{y}) > 0$. Then, either $\tilde G(\bar x, \bar y) > 0$ or $\|\nabla_y F(\bar x, \bar y)\|_{Y\ast} > 0$. In the case that $\tilde G(\bar x, \bar y) > 0$, let $\tilde{x} \in \arg\max_{x \in S}\left\{\nabla_x F(\bar x, \bar y)^T(\bar x - x)\right\}$ and define a direction $d \in \bbR^p \times \bbR^q$ by $d := (\tilde{x} - \bar{x}, 0)$. Then, $d$ is a feasible descent direction for (1) and therefore $(\bar{x}, \bar{y})$ is not locally optimal. Likewise, if $\|\nabla_y F(\bar x, \bar y)\|_{Y\ast} > 0$, let $\tilde{y} \in \arg\max\limits_{y \in \bbR^q}\{\nabla_y F(\bar x, \bar y)^T y : \|y \|_Y \le 1 \}$ and define $d := (0, -\tilde{y})$. Then $d$ is also a descent direction and therefore $(\bar{x}, \bar{y})$ is not locally optimal. 
\end{proof}

\subsection{Proof of Proposition 2.2}
\begin{proof}
By the gradient inequality for differentiable convex functions, it holds that:
\begin{align*}
F(\bar{x}, \bar{y}) - F^\ast &\leq \nabla_x F(\bar{x}, \bar{y})^T(\bar{x} - x^\ast) + \nabla_y F(\bar{x}, \bar{y})^T(\bar{y} - y^\ast) \\
&\leq \tilde{G}(\bar{x}, \bar{y}) + \|\nabla_y F(\bar{x}, \bar{y})\|_{Y\ast}\|\bar{y} - y^\ast\|_Y \\
&\leq \sqrt{\tfrac{\bar{C}}{2L_\nabla}}\cdot\tilde{G}(\bar{x}, \bar{y})\sqrt{\tfrac{\bar{2L_\nabla}}{\bar{C}}} + R\|\nabla_y F(\bar{x}, \bar{y})\|_{Y\ast} \\
&\leq \max\left\{\sqrt{\tfrac{\bar{C}}{2L_\nabla}}, R\right\} \cdot G(\bar{x}, \bar{y}) \ ,
\end{align*}
where the second inequality uses the definition of $\tilde{G}(\bar{x}, \bar{y})$ as well as H\"older's inequality.
\end{proof}

\subsection{Proof of Lemma 2.1}
\begin{proof}
Let $\sigma_S(\cdot)$ denote the support function of the set $S$, i.e., $\sigma_S(g) = \max_{x \in S}\left\{g^Tx\right\}$. Consider the function $\psi(\cdot, \cdot) : \bbR^p \times \bbR^q \to \bbR$ defined by $\psi(g, h) := (g^T\bar{x} + \sigma_S(-g))\sqrt{\frac{2L_\nabla}{\bar{C}}} ~+~ \|h\|_{Y\ast}$, which is a convex function of $(g, h)$. Note that $G(\bar{x}, \bar{y}) = \psi(\nabla_x F(\bar x, \bar y), \nabla_y F(\bar x, \bar y))$. Finally, Jensen's inequality yields:
\begin{equation*}
\bbE[\hat G] = \bbE[\psi(\hat g, \hat h)] \geq \psi(\nabla_x F(\bar x, \bar y), \nabla_y F(\bar x, \bar y)) = G(\bar{x}, \bar{y}) \ .
\end{equation*}
\end{proof}

\section{Proofs in Section 2.1}

\subsection{Useful Lemmas}\label{first_lem}

We use the following Lemma to prove the results in this section.

\begin{lemma}\label{var_lem}
Suppose that $(g_1, h_1), \ldots, (g_b, h_b)$ are i.i.d.\ random vectors in $\bbR^p \times \bbR^q$ with mean 0 and satisfying $\bbE[\|(g_i, h_i)\|_\ast^2] \leq \sigma^2$ for all $i = 1, \ldots, b$. Define $\hat g := \tfrac{1}{b}\sum_{i = 1}^b g_i$ and $\hat h = \tfrac{1}{b}\sum_{i = 1}^b h_i$.
Then, it holds that:
\begin{equation*}
\bbE[\|(\hat{g}, \hat{h})\|_\ast^2] ~\leq~ \frac{\kappa^2\sigma^2}{b} \ .
\end{equation*}
\end{lemma}
\begin{proof}
Recall that $\kappa_1 := \max_{(x,y) \ne 0} \|(x,y)\|_2/\|(x,y)\| = \max_{(s,t) \ne 0} \|(s,t)\|_\ast/\|(s,t)\|_2$ as well as $\kappa_2 := \max_{(x,y) \ne 0} \|(x,y)\|/\|(x,y)\|_2 = \max_{(s,t) \ne 0} \|(s,t)\|_2/\|(s,t)\|_\ast$. Hence, for any $(s,t) \in \bbR^p \times \bbR^q$, it holds that:
\begin{equation*}
\|(s,t)\|_\ast \leq \kappa_1\|(s,t)\|_2 \leq \kappa_1\kappa_2\|(s,t)\|_\ast = \kappa\|(s,t)\|_\ast \ .
\end{equation*}
Now we have that:
\begin{equation*}
\bbE[\|(\hat{g}, \hat{h})\|_\ast^2] \leq \kappa_1^2\cdot\bbE[\|(\hat{g}, \hat{h})\|_2^2] = \frac{\kappa_1^2}{b}\cdot\bbE[\|(g_1, h_1)\|_2^2] \leq \frac{\kappa^2}{b}\cdot\bbE[\|(g_1, h_1)\|_\ast^2] \leq \frac{\kappa^2\sigma^2}{b} \ ,
\end{equation*}
where the equality in the above chain uses the fact that $(g_1, h_1), \ldots, (g_b, h_b)$ are i.i.d.\ with mean 0.
\end{proof}

The proof of Theorem 2.1 is based on the following key lemma that bounds the expected progress per iteration.

\begin{lemma}\label{fw-sd-lemma}
For each $k \geq 0$, let $\calF_k$ denote the $\sigma$-field of all information gathered after completing iteration $k - 1$ of Algorithm 1, i.e., right before starting iteration $k$, and define $\Delta_k := 8L_\nabla(F(x_k, y_k) - F(x_{k+1}, y_{k+1}))$. Then, at every iteration $k \geq 0$, it holds that:
\begin{equation*}
\bbE[\Delta_k ~|~ \calF_k] ~\geq~ G(x_k, y_k)^2 - \frac{4\alpha_{\mathrm{AD}}\kappa^2\sigma^2}{b_k} \ .
\end{equation*}
\end{lemma}
{\em Proof.}
\paragraph{Case 1: AlternativeDirections = FALSE.}
Let us first consider the case of not using alternative directions, i.e., $\mathrm{AlternativeDirections} = $ FALSE.
By Assumption (A1), it is well-known and follows easily from the fundamental theorem of calculus that:
\begin{equation}\label{standard_smoothness}
F(x, y) \leq F(\bar{x}, \bar{y}) + \nabla F(\bar{x}, \bar{y})^T((x,y) - (\bar{x}, \bar{y})) + \tfrac{L_\nabla}{2}\|(x,y) - (\bar{x}, \bar{y})\|^2 \ \text{ for all } (x,y), (\bar{x}, \bar{y}) \in S \times \bbR^q \ .
\end{equation}
In the case of $\mathrm{AlternativeDirections} = $ FALSE, we have that $x_{k+1} = \bar{x}_{k} = x_k + \bar{\alpha}_k(\tilde{x}_k - x_k)$. 
Applying the above inequality to the iterates of Algorithm 1 yields deterministically: 
\begin{align*}
F(x_{k+1}, y_{k+1}) &\leq F(x_k, y_k) + \nabla F(x_k, y_k)^T((x_{k+1}, y_{k+1}) - (x_k, y_k)) + \tfrac{L_\nabla}{2}\|(x_{k+1}, y_{k+1}) - (x_k, y_k)\|^2 \\
&= F(x_k, y_k) + \nabla_x F(x_k, y_k)^T(x_{k+1} - x_k) + \tfrac{L_\nabla}{2}\|x_{k+1} - x_k\|_X^2 \\
& \ \ \ \ \ \ \ \ \ \ \ \ \ \ \ \ + \nabla_y F(x_k, y_k)^T(y_{k+1} - y_k) + \tfrac{L_\nabla}{2}\|y_{k+1} - y_k\|_Y^2 \\
&= F(x_k, y_k) + \bar{\alpha}_k\nabla_x F(x_k, y_k)^T(\tilde{x}_k - x_k) + \tfrac{L_\nabla\bar{\alpha}_k^2}{2}\|\tilde{x}_k - x_k\|_X^2 \\
& \ \ \ \ \ \ \ \ \ \ \ \ \ \ \ \ -\alpha_k\nabla_y F(x_k, y_k)^T\tilde{y}_k + \tfrac{L_\nabla\alpha_k^2}{2}\|\tilde{y}_k\|_Y^2 \\
&\leq F(x_k, y_k) + \bar{\alpha}_k\nabla_x F(x_k, y_k)^T(\tilde{x}_k - x_k) + \tfrac{L_{\nabla}\text{diam}(S)^2\bar{\alpha}_k^2}{2} \\
& \ \ \ \ \ \ \ \ \ \ \ \ \ \ \ \ -\alpha_k\nabla_y F(x_k, y_k)^T\tilde{y}_k + \tfrac{L_\nabla\alpha_k^2}{2} \\
&= F(x_k, y_k) + \bar{\alpha}_k\hat{g}_k^T(\tilde{x}_k - x_k) + \bar{\alpha}_k(\nabla_x F(x_k, y_k) - \hat{g}_k)^T(\tilde{x}_k - x_k) + \tfrac{L_{\nabla}\text{diam}(S)^2\bar{\alpha}_k^2}{2} \\
& \ \ \ \ \ \ \ \ \ \ \ \ \ \ \ \ -\alpha_k\hat{h}_k^T\tilde{y}_k + \alpha_k(\hat{h}_k - \nabla_y F(x_k, y_k))^T\tilde{y}_k + \tfrac{L_\nabla\alpha_k^2}{2} \\
&= F(x_k, y_k) - \bar{\alpha}_k\tilde{G}_k + \bar{\alpha}_k(\nabla_x F(x_k, y_k) - \hat{g}_k)^T(\tilde{x}_k - x_k) + \tfrac{L_{\nabla}\text{diam}(S)^2\bar{\alpha}_k^2}{2} \\
& \ \ \ \ \ \ \ \ \ \ \ \ \ \ \ \ -\alpha_k\|\hat{h}_k\|_{Y\ast} + \alpha_k(\hat{h}_k - \nabla_y F(x_k, y_k))^T\tilde{y}_k + \tfrac{L_\nabla\alpha_k^2}{2} \ .
\end{align*}
Recall that for any $\gamma > 0$ and vectors $s, x \in \bbR^p$, it holds that $s^Tx \leq \tfrac{1}{2\gamma}\|s\|_{X\ast}^2 + \tfrac{\gamma}{2}\|x\|_X^2$. Applying this inequality with $\gamma \gets L_\nabla$, $s \gets \nabla_x F(x_k, y_k) - \hat{g}_k$ and $x \gets \bar{\alpha}_k(\tilde{x}_k - x_k)$ yields:
\begin{align*}
F(x_{k+1}, y_{k+1}) &\leq F(x_k, y_k) - \bar{\alpha}_k\tilde{G}_k + \tfrac{1}{2L_\nabla}\|\nabla_x F(x_k, y_k) - \hat{g}_k\|_{X\ast}^2 + \tfrac{L_\nabla\bar{\alpha}_k^2}{2}\|\tilde{x}_k - x_k\|_X^2 + \tfrac{L_{\nabla}\text{diam}(S)^2\bar{\alpha}_k^2}{2} \\
& \ \ \ \ \ \ \ \ \ \ \ \ \ \ \ \ -\alpha_k\|\hat{h}_k\|_{Y\ast} + \alpha_k(\hat{h}_k - \nabla_y F(x_k, y_k))^T\tilde{y}_k + \tfrac{L_\nabla\alpha_k^2}{2} \\
&\leq F(x_k, y_k) - \bar{\alpha}_k\tilde{G}_k + \tfrac{1}{2L_\nabla}\|\nabla_x F(x_k, y_k) - \hat{g}_k\|_{X\ast}^2 + \tfrac{\bar{C}\bar{\alpha}_k^2}{2} \\
& \ \ \ \ \ \ \ \ \ \ \ \ \ \ \ \ -\alpha_k\|\hat{h}_k\|_{Y\ast} + \alpha_k(\hat{h}_k - \nabla_y F(x_k, y_k))^T\tilde{y}_k + \tfrac{L_\nabla\alpha_k^2}{2} \ ,
\end{align*}
where the second inequality uses $\bar{C} \geq 2L_\nabla\cdot\diam(S)^2$. Applying the same reasoning on the space of $y$ variables with norms $\|\cdot\|_Y$ and $\|\cdot\|_{Y\ast}$ yields:
\begin{align*}
F(x_{k+1}, y_{k+1}) &\leq F(x_k, y_k) - \bar{\alpha}_k\tilde{G}_k + \tfrac{1}{2L_\nabla}\|\nabla_x F(x_k, y_k) - \hat{g}_k\|_{X\ast}^2 + \tfrac{\bar{C}\bar{\alpha}_k^2}{2} \\
& \ \ \ \ \ \ \ \ \ \ \ \ \ \ \ \ -\alpha_k\|\hat{h}_k\|_{Y\ast} + \tfrac{1}{2L_\nabla}\|\nabla_y F(x_k, y_k) -\hat{h}_k\|_{Y\ast}^2 + \tfrac{L_\nabla\alpha_k^2}{2}\|\tilde{y}_k\|_Y^2 + \tfrac{L_\nabla\alpha_k^2}{2} \\
F(x_{k+1}, y_{k+1}) &\leq F(x_k, y_k) - \bar{\alpha}_k\tilde{G}_k + \tfrac{1}{2L_\nabla}\|\nabla_x F(x_k, y_k) - \hat{g}_k\|_{X\ast}^2 + \tfrac{\bar{C}\bar{\alpha}_k^2}{2} \\
& \ \ \ \ \ \ \ \ \ \ \ \ \ \ \ \ -\alpha_k\|\hat{h}_k\|_{Y\ast} + \tfrac{1}{2L_\nabla}\|\nabla_y F(x_k, y_k) -\hat{h}_k\|_{Y\ast}^2 + L_\nabla\alpha_k^2 \ ,
\end{align*}
where the second inequality uses $\|\tilde{y}_k\|_Y \leq 1$.
Using $\bar\alpha_k = \tilde{G}_k/\bar{C}$, and $\alpha_k = \|\hat{h}_k\|_{Y\ast}/2L_{\nabla}$ yields:
\begin{align*}
F(x_{k+1}, y_{k+1}) &\leq F(x_k, y_k) - \frac{\tilde{G}_k^2}{2\bar{C}} - \frac{\|\hat{h}_k\|_{Y\ast}^2}{4L_{\nabla}} + \frac{1}{2L_\nabla}\|\nabla_x F(x_k, y_k) - \hat{g}_k\|_{X\ast}^2 + \frac{1}{2L_\nabla}\|\nabla_y F(x_k, y_k) -\hat{h}_k\|_{Y\ast}^2
\end{align*}
Multiplying the above inequality by $8L_{\nabla}$ and rearranging terms yields:
\begin{align*}
\Delta_k &\geq \frac{4L_\nabla\tilde{G}_k^2}{\bar{C}} + 2\|\hat{h}_k\|_{Y\ast}^2 - 4\|\nabla_x F(x_k, y_k) - \hat{g}_k\|_{X\ast}^2 - 4\|\nabla_y F(x_k, y_k) -\hat{h}_k\|_{Y\ast}^2 \\
&= \frac{4L_\nabla\tilde{G}_k^2}{\bar{C}} + 2\|\hat{h}_k\|_{Y\ast}^2 -4\|(\nabla_x F(x_k, y_k), \nabla_y F(x_k, y_k)) - (\hat{g}_k, \hat{h}_k)\|_\ast^2 \\
&\geq \left(\tilde{G}_k\sqrt{\frac{2L_\nabla}{\bar{C}}} + \|\hat{h}\|_{Y\ast}\right)^2 - 4\|(\nabla_x F(x_k, y_k), \nabla_y F(x_k, y_k)) - (\hat{g}_k, \hat{h}_k)\|_\ast^2 \ ,
\end{align*}
where the second inequality uses $(a + b)^2 \leq 2(a^2 + b^2)$. By combining assumption (A3) with Lemma \ref{var_lem}, we have that
\begin{equation*}
\bbE\left[\|(\nabla_x F(x_k, y_k), \nabla_y F(x_k, y_k)) - (\hat{g}_k, \hat{h}_k)\|_\ast^2 ~|~ \calF_k\right] \leq \frac{\kappa^2\sigma^2}{b_k} \ .
\end{equation*}
Furthermore, by combining Lemma 2.1 with Jensen's inequality on $t \mapsto t^2$ we have:
\begin{equation*}
G(x_k, y_k)^2 \leq \left(\bbE\left[\tilde{G}_k\sqrt{\tfrac{2L_\nabla}{\bar{C}}} + \|\hat{h}\|_{Y\ast} ~|~ \calF_k\right]\right)^2 \leq \bbE\left[\left(\tilde{G}_k\sqrt{\tfrac{2L_\nabla}{\bar{C}}} + \|\hat{h}\|_{Y\ast}\right)^2 ~|~ \calF_k\right] \ .
\end{equation*}
Combining the previous inequalities together yields:
\begin{equation*}
\bbE[\Delta_k ~|~ \calF_k] ~\geq~ G(x_k, y_k)^2 - \frac{4\kappa^2\sigma^2}{b_k} \ ,
\end{equation*} 
which proves the result for Case 1.

\paragraph{Case 2: AlternativeDirections = TRUE.}
First notice that we can decompose $\Delta_k$ as:
\begin{equation}\label{decompose}
\Delta_k ~=~ 8L_{\nabla}(F(x_{k}, y_{k}) - F(\bar{x}_k, y_{k+1})) ~+~ 8L_{\nabla}(F(\bar{x}_k, y_{k+1}) - F(x_{k+1}, y_{k+1})) \ .
\end{equation}
By the exact same reasoning as above, we have that
\begin{equation}\label{old_fw_bound}
\bbE[8L_{\nabla}(F(x_{k}, y_{k}) - F(\bar{x}_k, y_{k+1})) ~|~ \calF_k] ~\geq~ G(x_k, y_k)^2 - \frac{4\kappa^2\sigma^2}{b_k} \ .
\end{equation}

Let $\calG_k$ denote the $\sigma$-field of all information gathered after completing Step (3.) of iteration $k$ of Algorithm 2, i.e., right before starting Step (4.) (the alternative direction step). Note that $\calF_k \subset \calG_k$. Applying \eqref{standard_smoothness} at Step (4.) of Algorithm 2, we have deterministically:
\begin{align*}
F(x_{k+1}, y_{k+1}) &\leq F(\bar{x}_k, y_{k+1}) + \nabla F(\bar{x}_k, y_{k+1})^T((x_{k+1}, y_{k+1}) - (\bar{x}_k, y_{k+1})) + \tfrac{L_\nabla}{2}\|(x_{k+1}, y_{k+1}) - (\bar{x}_k, y_{k+1})\|^2 \\
&= F(\bar{x}_k, y_{k+1}) + \nabla_x F(\bar{x}_k, y_{k+1})^T(x_{k+1} - \bar{x}_k) + \tfrac{L_\nabla}{2}\|x_{k+1} - \bar{x}_k\|_X^2 \\
&= F(\bar{x}_k, y_{k+1}) + \bar{\beta}_k\nabla_x F(\bar{x}_k, y_{k+1})^Td_k + \tfrac{L_\nabla\bar{\beta}_k^2}{2}\|d_k\|_X^2 \\
&\leq F(\bar{x}_k, y_{k+1}) + \bar{\beta}_k\nabla_x F(\bar{x}_k, y_{k+1})^Td_k + \tfrac{L_\nabla\diam(S)^2\bar{\beta}_k^2}{2} \\
&= F(\bar{x}_k, y_{k+1}) + \bar{\beta}_k\check{g}_k^Td_k + \bar{\beta}_k(\nabla_x F(\bar{x}_k, y_{k+1}) - \check{g}_k)^Td_k + \tfrac{L_\nabla\diam(S)^2\bar{\beta}_k^2}{2} \\
&= F(\bar{x}_k, y_{k+1}) - \bar{\beta}_kA_k + \bar{\beta}_k(\nabla_x F(\bar{x}_k, y_{k+1}) - \check{g}_k)^Td_k + \tfrac{L_\nabla\diam(S)^2\bar{\beta}_k^2}{2}
\end{align*}
Applying the inequality $s^Tx \leq \tfrac{1}{2\gamma}\|s\|_{X\ast}^2 + \tfrac{\gamma}{2}\|x\|_X^2$ with $\gamma \gets L_\nabla$, $s \gets \nabla_x F(\bar{x}_k, y_{k+1}) - \check{g}_k$ and $x \gets \bar{\beta}_kd_k$ yields:
\begin{align*}
F(x_{k+1}, y_{k+1}) &\leq F(\bar{x}_k, y_{k+1}) - \bar{\beta}_kA_k + \tfrac{1}{2L_{\nabla}}\|\nabla_x F(\bar{x}_k, y_{k+1}) - \check{g}_k\|_{X\ast}^2 + \tfrac{L_{\nabla}\bar{\beta}_k^2}{2}\|d_k\|_X^2 + \tfrac{L_\nabla\diam(S)^2\bar{\beta}_k^2}{2} \\
&\leq F(\bar{x}_k, y_{k+1}) - \bar{\beta}_kA_k + \tfrac{1}{2L_{\nabla}}\|\nabla_x F(\bar{x}_k, y_{k+1}) - \check{g}_k\|_{X\ast}^2 + \tfrac{\bar{C}\bar{\beta}_k^2}{2} \ ,
\end{align*}
where the second inequality uses $\bar{C} \geq 2L_\nabla\cdot\diam(S)^2$. Notice that $\bar{\beta}_k = \min\left\{A_k/\bar{C}, \alpha^{\text{stop}}_k\right\}$ minimizes the quadratic function $\beta \mapsto -\beta A_k + \tfrac{\bar{C}\beta^2}{2}$ on the interval $[0, \alpha^{\text{stop}}_k]$. Hence, in particular we have that $- \bar{\beta}_kA_k + \tfrac{\bar{C}\bar{\beta}_k^2}{2} \leq 0$ and therefore:
\begin{equation}\label{therefore_zzz}
F(x_{k+1}, y_{k+1}) ~\leq~ F(\bar{x}_k, y_{k+1}) + \tfrac{1}{2L_{\nabla}}\|\nabla_x F(\bar{x}_k, y_{k+1}) - \check{g}_k\|_{X\ast}^2 \ .
\end{equation}
Multiplying the above inequality by $8L_{\nabla}$ and rearranging terms yields:
\begin{align*}
8L_{\nabla}(F(\bar{x}_k, y_{k+1}) - F(x_{k+1}, y_{k+1})) ~&\geq~ -4\|\nabla_x F(\bar{x}_k, y_{k+1}) - \check{g}_k\|_{X\ast}^2 \\
~&=~ -4\|(\nabla_x F(\bar{x}_k, y_{k+1}), \nabla_y F(\bar{x}_k, y_{k+1})) - (\check{g}_k, \nabla_y F(\bar{x}_k, y_{k+1}))\|_{\ast}^2 \ .
\end{align*}
Using the definition of the dual norm $\|\cdot\|_\ast$ as well as assumption (A3), we have for all $(x,y) \in S \times \bbR^q$ that:
\begin{align*}
&\bbE_{z \sim D}\left[\|(\nabla_x f(x, y, z), \nabla_y F(x,y)) - (\nabla_x F(x,y), \nabla_y F(x,y))\|_\ast^2\right] = \\
&\bbE_{z \sim D}\left[\|(\nabla_x f(x, y, z) - \nabla_x F(x,y)\|_{X\ast}^2 + \|0\|_{Y\ast}^2\right] \leq \\
&\bbE_{z \sim D}\left[\|(\nabla_x f(x, y, z) - \nabla_x F(x,y)\|_{X\ast}^2 + \|\nabla_y f(x, y, z) - \nabla_y F(x,y) \|_{Y\ast}^2\right] = \\
&\bbE_{z \sim D}\left[\|\nabla f(x, y, z) - \nabla F(x, y)\|_\ast^2\right] \leq \sigma^2
\end{align*}
Hence, by combining the above with Lemma \ref{var_lem}, we have that
\begin{equation*}
\bbE\left[\|(\nabla_x F(\bar{x}_k, y_{k+1}), \nabla_y F(\bar{x}_k, y_{k+1})) - (\check{g}_k, \nabla_y F(\bar{x}_k, y_{k+1}))\|_{\ast}^2 ~|~ \calG_k\right] \leq \frac{\kappa^2\sigma^2}{b_k} \ .
\end{equation*}
Combining the previous inequalities together yields:
\begin{equation*}
\bbE[8L_{\nabla}(F(\bar{x}_k, y_{k+1}) - F(x_{k+1}, y_{k+1})) ~|~ \calG_k] ~\geq~ - \frac{4\kappa^2\sigma^2}{b_k} \ .
\end{equation*}
Using the tower property of conditional expectation we have that
\begin{equation*}
\bbE[8L_{\nabla}(F(\bar{x}_k, y_{k+1}) - F(x_{k+1}, y_{k+1})) ~|~ \calF_k] ~=~ \bbE\left[\bbE[8L_{\nabla}(F(\bar{x}_k, y_{k+1}) - F(x_{k+1}, y_{k+1})) ~|~ \calG_k] ~|~ \calF_k\right] ~\geq~ -\frac{4\kappa^2\sigma^2}{b_k} \ .
\end{equation*}
Finally combining the above with \eqref{old_fw_bound} and and \eqref{decompose} yields:
\begin{equation*}
\bbE[\Delta_k ~|~ \calF_k] ~\geq~ G(x_k, y_k)^2 - \frac{8\kappa^2\sigma^2}{b_k} \ ,
\end{equation*}
which proves the result in Case 2.
\qed

\subsection{Proof of Theorem 2.1}\label{big_theorem}
By combining Lemma \ref{fw-sd-lemma} with the law of iterated expectations, it holds for each $k \in \{0, \ldots, K\}$ that:
\begin{equation*}
\bbE[\Delta_k] ~=~ \bbE\left[\bbE[\Delta_k ~|~ \calF_k]\right] ~\geq~ \bbE[G(x_k, y_k)^2] - \frac{4\alpha_{\mathrm{AD}}\kappa^2\sigma^2}{b_k} \ .
\end{equation*}
Recalling that $\bbE[\Delta_k] = 8L_{\nabla}\bbE[F(x_k, y_k)] - 8L_{\nabla}\bbE[F(x_{k+1}, y_{k+1})]$ and summing the above inequality over all $k \in \{0, \ldots, K\}$ yields:
\begin{equation*}
\sum_{k = 0}^K\bbE[G(x_k, y_k)^2] ~\leq~ 8L_{\nabla}(F(x_0, y_0) - \bbE[F(x_{K+1}, y_{K+1})]) ~+~ 4\alpha_{\mathrm{AD}}\kappa^2\sigma^2\sum_{k = 0}^K \frac{1}{b_k} \ .
\end{equation*}
Then, using $F^\ast \leq \bbE[F(x_{K+1}, y_{K+1})]$ and dividing by $K+1$ yields:
\begin{equation*}
\frac{1}{K+1}\sum_{k = 0}^K\bbE[G(x_k, y_k)^2] ~\leq~ \frac{8L_{\nabla}(F(x_0, y_0) - F^\ast)}{K+1} ~+~ \frac{4\alpha_{\mathrm{AD}}\kappa^2\sigma^2}{K+1}\sum_{k = 0}^K \frac{1}{b_k} \ .
\end{equation*}
Finally, since $(\hat x_k, \hat y_k)$ is chosen uniformly at random from $(x_0, y_0), \ldots, (x_{K}, y_K)$, another iterated expectations argument implies that $\bbE[G(\hat{x}_K, \hat{y}_K)^2] = \frac{1}{K+1}\sum_{k = 0}^K\bbE[G(x_k, y_k)^2]$, from which the desired result follows.\qed 

\section{Example of In-Face Direction Computation}
In this section, we briefly describe how to compute an in-face direction in the case where $S = \{x : \|x\|_1 \leq \delta\}$ is an $\ell_1$-ball.
Let $\bar{x} \in S$ be a given point representing our current iterate. In particular, let us discuss the complexity of a solving a linear optimization problem $\min_{x \in {\cal F}(\bar{x})}c^Tx$ over the minimal face ${\cal F}(\bar{x})$ containing $\bar{x}$ for some given $c \in \bbR^p$, which is required in the ``away step" direction (3), for example. 

Let us consider two cases:  {\em (i)} $\bar{x} \in \text{int}(S)$ and {\em (ii)} $\bar{x} \in \partial S$, where $\partial S$ represents the boundary of $S$. In case {\em (i)}, we simply have that ${\cal F}(\bar{x}) = S$ and the linear optimization problem is simply that of minimizing $c^Tx$ over $S$, which is the same subproblem as the Frank-Wolfe step as is equivalent to computing $\|c\|_\infty = \max_{j = 1, \ldots, p} |c_j|$. Otherwise, if $\bar{x} \in \partial S$, then we have that $\|\bar{x}\|_1 = \delta$ and let $J_+(\bar{x}) = \{j : \bar{x}_j > 0\}$, $J_-(\bar{x}) = \{j : \bar{x}_j < 0\}$, $J_0(\bar{x}) = \{j : \bar{x}_j = 0\}$. Then, it is straightforward to see that
\begin{align*}
{\cal F}(\bar{x}) = \{x :~ &x_j > 0 \text{ if } j \in J_+(\bar{x}), \\
&x_j < 0 \text{ if } j \in J_-(\bar{x}), \\
&x_j = 0 \text{ if } j \in J_0(\bar{x}), \\
&\sum_{j \in J_+(\bar{x})} x_j - \sum_{j \in J_-(\bar{x})} x_j = \delta\} \ .
\end{align*}
(Note that we clearly have $\|x\|_1 = \sum_{j \in J_+(\bar{x})} x_j - \sum_{j \in J_-(\bar{x})} x_j$ in the above.) Then, in order to solve $\min_{x \in {\cal F}(\bar{x})}c^Tx$, we can simply follow an argument that enumerates the extreme points of the above polytope, from which we obtain that:
\begin{equation*}
j^\ast \in \argmin_{j \in J_+(\bar{x}) \cup J_-(\bar{x})} \sgn(\bar{x}_j)c_j \ \Longrightarrow \ \sgn(\bar{x}_{j^\ast})\delta e_{j^\ast} \in \arg\min_{x \in {\cal F}(\bar{x})}c^Tx \ .
\end{equation*}
Thus, as in the case when $\bar{x} \in \text{int}(S)$, we can solve $\min_{x \in {\cal F}(\bar{x})}c^Tx$ efficiently in time that is linear in $p$.

\section{Proofs in Section 3}

Let us first state and prove the following lemma, which is the ``block coordinate'' version of Lemma \ref{fw-sd-lemma} and will be critical proving Theorem 3.1.

\begin{lemma}\label{block-alt-fw-sd-lemma}
For each $k \geq 0$, let $\calF_k$ denote the $\sigma$-field of all information gathered after completing iteration $k - 1$ of the Block Coordinate variant of Algorithm 1, i.e., right before starting iteration $k$, and define $\Delta_k := 8L_\nabla(F(x_k, y_k) - F(x_{k+1}, y_{k+1}))$. Then, at every iteration $k \geq 0$, it holds that:
\begin{equation*}
\bbE[\Delta_k ~|~ \calF_k] ~\geq~ G(x_k, y_k)^2 - \frac{4\alpha_{\mathrm{AD}}\kappa^2\sigma^2}{b_k} \ .
\end{equation*}
where the modified Frank-Wolfe gap $G(\cdot, \cdot)$ (Definition 2.1) is defined using $\bar{C} := \sum_{i = 1}^N \bar{C}_i$.
\end{lemma}
\begin{proof}
{\bf Case 1: AlternativeDirections = FALSE.}
Let us define $\Theta_k := 8L_{\nabla}(F(x_{k}, y_{k}) - F(\bar{x}_k, y_{k+1}))$. We first bound $\Theta_k$ following the same general structure as in the proof of Lemma \ref{fw-sd-lemma} in Section \ref{first_lem}.

Applying \eqref{standard_smoothness} at Steps (2.)/(3.) of the block coordinate version of Algorithm 2, we have deterministically:
\begin{align}\label{first_block_eqn}
F(\bar{x}_{k}, y_{k+1}) &\leq F(x_k, y_k) + \nabla F(x_k, y_k)^T((\bar{x}_{k}, y_{k+1}) - (x_k, y_k)) + \tfrac{L_\nabla}{2}\|(\bar{x}_{k}, y_{k+1}) - (x_k, y_k)\|^2 \nonumber \\
&= F(x_k, y_k) + \nabla_x F(x_k, y_k)^T(\bar{x}_{k} - x_k) + \tfrac{L_\nabla}{2}\|\bar{x}_{k} - x_k\|_X^2 \\
& \ \ \ \ \ \ \ \ \ \ \ \ \ \ \ \ + \nabla_y F(x_k, y_k)^T(y_{k+1} - y_k) + \tfrac{L_\nabla}{2}\|y_{k+1} - y_k\|_Y^2 \nonumber
\end{align}
For ease of notation, define $\Gamma_k := \nabla_x F(x_k, y_k)^T(\bar{x}_{k} - x_k) + \tfrac{L_\nabla}{2}\|\bar{x}_{k} - x_k\|_X^2$. Utilizing the block coordinate structure, we have that:
\begin{align*}
\Gamma_k &= \nabla_x F(x_k, y_k)^T(\bar{x}_{k} - x_k) + \tfrac{L_\nabla}{2}\|\bar{x}_{k} - x_k\|_X^2 \\
&= \sum_{i = 1}^N \nabla_x^{(i)} F(x_k, y_k)^T(\bar{x}_{k}^{(i)} - x_k^{(i)}) ~+~ \tfrac{L_\nabla}{2}\sum_{i = 1}^N \|\bar{x}_{k}^{(i)} - x_k^{(i)}\|_{X,i}^2 \\
&= \sum_{i = 1}^N\left[\nabla_x^{(i)} F(x_k, y_k)^T(\bar{x}_{k}^{(i)} - x_k^{(i)}) + \tfrac{L_\nabla}{2}\|\bar{x}_{k}^{(i)} - x_k^{(i)}\|_{X,i}^2 \right] \\
&= \sum_{i = 1}^N\left[\bar{\alpha}_k^i\nabla_x^{(i)} F(x_k, y_k)^T(\tilde{x}_k^{(i)} - x_k^{(i)}) + \tfrac{L_\nabla(\bar{\alpha}_k^i)^2}{2}\|\tilde{x}_{k}^{(i)} - x_k^{(i)}\|_{X,i}^2 \right] \\
&\leq \sum_{i = 1}^N\left[\bar{\alpha}_k^i\nabla_x^{(i)} F(x_k, y_k)^T(\tilde{x}_k^{(i)} - x_k^{(i)}) + \tfrac{L_\nabla\diam(S_i)^2(\bar{\alpha}_k^i)^2}{2} \right] \\
&= \sum_{i = 1}^N\left[\bar{\alpha}_k^i(\hat{g}_k^{(i)})^T(\tilde{x}_k^{(i)} - x_k^{(i)}) + \bar{\alpha}_k^i(\nabla_x^{(i)}F(x_k, y_k) - \hat{g}_k^{(i)})^T(\tilde{x}_k^{(i)} - x_k^{(i)}) + \tfrac{L_\nabla\diam(S_i)^2(\bar{\alpha}_k^i)^2}{2} \right] \\
&= \sum_{i = 1}^N\left[-\bar{\alpha}_k^i\tilde{G}_k^i + \bar{\alpha}_k^i(\nabla_x^{(i)}F(x_k, y_k) - \hat{g}_k^{(i)})^T(\tilde{x}_k^{(i)} - x_k^{(i)}) + \tfrac{L_\nabla\diam(S_i)^2(\bar{\alpha}_k^i)^2}{2} \right] \\
&\leq \sum_{i = 1}^N\left[-\bar{\alpha}_k^i\tilde{G}_k^i + \tfrac{1}{2L_{\nabla}}\|\nabla_x^{(i)}F(x_k, y_k) - \hat{g}_k^{(i)}\|_{X\ast,i}^2 + \tfrac{L_{\nabla}(\bar{\alpha}_k^i)^2}{2}\|\tilde{x}_{k}^{(i)} - x_k^{(i)}\|_{X,i}^2 + \tfrac{L_\nabla\diam(S_i)^2(\bar{\alpha}_k^i)^2}{2} \right] \\
&\leq \sum_{i = 1}^N\left[-\bar{\alpha}_k^i\tilde{G}_k^i + \tfrac{1}{2L_{\nabla}}\|\nabla_x^{(i)}F(x_k, y_k) - \hat{g}_k^{(i)}\|_{X\ast,i}^2 + \tfrac{\bar{C}_i(\bar{\alpha}_k^i)^2}{2} \right] \\
&= \sum_{i = 1}^N\left[-\bar{\alpha}_k^i\tilde{G}_k^i + \tfrac{\bar{C}_i(\bar{\alpha}_k^i)^2}{2} \right] ~+~ + \tfrac{1}{2L_{\nabla}}\|\nabla_x F(x_k, y_k) - \hat{g}_k\|_{X\ast}^2 \ .
\end{align*}
The second as well as the final equality above uses the Euclidean structure of $\|\cdot\|_X$ and $\|\cdot\|_{X\ast}$, namely $\|x\|_X^2 := \sum_{i = 1}^N \|x^{(i)}\|_{X,i}^2$ and $\|s\|_{X\ast}^2 := \sum_{i = 1}^N \|s^{(i)}\|_{X\ast,i}^2$. The second inequality uses $(s^{(i)})^Tx^{(i)} \leq \tfrac{1}{2\gamma}\|s^{(i)}\|_{X\ast,i}^2 + \tfrac{\gamma}{2}\|x^{(i)}\|_{X,i}^2$ with $\gamma \gets L_\nabla$, $s \gets \nabla_x^{(i)} F(x_k, y_k) - \hat{g}_k^{(i)}$ and $x \gets \bar{\alpha}_k^i(\tilde{x}_k^{(i)} - x_k^{(i)})$ for each $i \in \{1, \ldots, N\}$, and the third inequality uses $\bar{C}_i \geq 2L_{\nabla}\diam(S_i)^2$. Now, combining \eqref{first_block_eqn} with the above as well as the same reasoning on the space of $y$ variables that was used in the proof of Lemma 2.2 yields:
\begin{align*}
F(\bar{x}_{k}, y_{k+1}) &\leq F(x_k, y_k) + \sum_{i = 1}^N\left[-\bar{\alpha}_k^i\tilde{G}_k^i + \tfrac{\bar{C}_i(\bar{\alpha}_k^i)^2}{2} \right] +  \tfrac{1}{2L_{\nabla}}\|\nabla_x F(x_k, y_k) - \hat{g}_k\|_{X\ast}^2 \\
& \ \ \ \ \ \ \ \ \ \ \ \ \ \ \ \ -\alpha_k\|\hat{h}_k\|_{Y\ast} + \tfrac{1}{2L_\nabla}\|\nabla_y F(x_k, y_k) -\hat{h}_k\|_{Y\ast}^2 + L_\nabla\alpha_k^2 \ .
\end{align*}
Using $\bar\alpha_k^{i} = \tilde{G}_k^i/\bar{C}_i$, and $\alpha_k = \|\hat{h}_k\|_{Y\ast}/2L_{\nabla}$ yields:
\begin{align*}
F(\bar{x}_{k}, y_{k+1}) &\leq F(x_k, y_k) - \sum_{i = 1}^N \frac{(\tilde{G}_k^{i})^2}{2\bar{C}_i} - \frac{\|\hat{h}_k\|_{Y\ast}^2}{4L_{\nabla}} + \frac{1}{2L_\nabla}\|\nabla_x F(x_k, y_k) - \hat{g}_k\|_{X\ast}^2 + \frac{1}{2L_\nabla}\|\nabla_y F(x_k, y_k) -\hat{h}_k\|_{Y\ast}^2 \\
&= F(x_k, y_k) - \sum_{i = 1}^N \frac{(\tilde{G}_k^{i})^2}{2\bar{C}_i} - \frac{\|\hat{h}_k\|_{Y\ast}^2}{4L_{\nabla}} + \frac{1}{2L_\nabla}\|(\nabla_x F(x_k, y_k), \nabla_y F(x_k, y_k)) - (\hat{g}_k, \hat{h}_k)\|_\ast^2 \ .
\end{align*}
Letting $\Theta_k := 8L_{\nabla}(F(x_k, y_k) - F(\bar{x}_{k}, y_{k+1}))$ and multiplying the above inequality by $8L_{\nabla}$ and rearranging terms yields:
\begin{equation}\label{blah_blah}
\Theta_k \geq 4L_{\nabla}\sum_{i = 1}^N\frac{(\tilde{G}_k^{i})^2}{\bar{C}_i} + 2\|\hat{h}_k\|_{Y\ast}^2 - 4\|(\nabla_x F(x_k, y_k), \nabla_y F(x_k, y_k)) - (\hat{g}_k, \hat{h}_k)\|_\ast^2 \ .
\end{equation}
Recall that for any two sequences $\{g_i\}_{i = 1}^N$ and $\{c_i\}_{i = 1}^N$ with $g_i \geq 0$ and $c_i > 0$, Cauchy-Schwartz yields:
\begin{equation*}
\left(\sum_{i = 1}^N g_i\right)^2 = \left(\sum_{i = 1}^N \frac{g_i\sqrt{c_i}}{\sqrt{c_i}}\right)^2 ~\leq~ \left(\sum_{i = 1}^N \frac{g_i^2}{c_i}\right)\left(\sum_{i = 1}^N c_i\right) \ .
\end{equation*}
Recall that $\bar{C} = \sum_{i = 1}^N \bar{C}_i$ and let us define $\tilde{G}_k := \sum_{i = 1}^N \tilde{G}_k^i$. 
Then, applying the above to \eqref{blah_blah} with $g_i \gets \tilde{G}_k^i$ and $c_i \gets \bar{C}_i$ yields:
\begin{align*}
\Theta_k &\geq \frac{4L_\nabla\tilde{G}_k^2}{\bar{C}} + 2\|\hat{h}_k\|_{Y\ast}^2 -4\|(\nabla_x F(x_k, y_k), \nabla_y F(x_k, y_k)) - (\hat{g}_k, \hat{h}_k)\|_\ast^2 \\
&\geq \left(\tilde{G}_k\sqrt{\frac{2L_\nabla}{\bar{C}}} + \|\hat{h}_k\|_{Y\ast}\right)^2 - 4\|(\nabla_x F(x_k, y_k), \nabla_y F(x_k, y_k)) - (\hat{g}_k, \hat{h}_k)\|_\ast^2 \ ,
\end{align*}
where the second inequality uses $(a + b)^2 \leq 2(a^2 + b^2)$. By combining assumption (A3) with Lemma \ref{var_lem}, we have that
\begin{equation*}
\bbE\left[\|(\nabla_x F(x_k, y_k), \nabla_y F(x_k, y_k)) - (\hat{g}_k, \hat{h}_k)\|_\ast^2 ~|~ \calF_k\right] \leq \frac{\kappa^2\sigma^2}{b_k} \ .
\end{equation*}
Furthermore, note that the decomposable structure of $S$ implies that $\tilde{G}_k = \max_{x \in S}\left\{\hat{g}_k^T(\bar{x} - x)\right\}$. Therefore, we may apply Lemma 2.1 along with Jensen's inequality on $t \mapsto t^2$ to yield:
\begin{equation*}
G(x_k, y_k)^2 \leq \left(\bbE\left[\tilde{G}_k\sqrt{\tfrac{2L_\nabla}{\bar{C}}} + \|\hat{h}\|_{Y\ast} ~|~ \calF_k\right]\right)^2 \leq \bbE\left[\left(\tilde{G}_k\sqrt{\tfrac{2L_\nabla}{\bar{C}}} + \|\hat{h}\|_{Y\ast}\right)^2 ~|~ \calF_k\right] \ .
\end{equation*}
Combining the previous inequalities together yields:
\begin{equation*}
\bbE[\Theta_k ~|~ \calF_k] ~\geq~ G(x_k, y_k)^2 - \frac{4\kappa^2\sigma^2}{b_k} \ ,
\end{equation*}
which proves the result in Case 1 since $x_{k+1} = \bar{x}_k$ in that case.

{\bf Case 2: AlternativeDirections = TRUE.}
Now notice that we can decompose $\Delta_k$ as:
\begin{equation}\label{decompose_block}
\Delta_k ~=~ 8L_{\nabla}(F(x_{k}, y_{k}) - F(\bar{x}_k, y_{k+1})) ~+~ 8L_{\nabla}(F(\bar{x}_k, y_{k+1}) - F(x_{k+1}, y_{k+1})) \ .
\end{equation}
Let us define $\Lambda_k := 8L_{\nabla}(F(\bar{x}_k, y_{k+1}) - F(x_{k+1}, y_{k+1}))$ so that $\Delta_k = \Theta_k + \Lambda_k$.

Let us now work on bounding $\Lambda_k$ using the same general structure as that of Case 2 of Lemma \ref{fw-sd-lemma} in Section \ref{first_lem}. Let $\calG_k$ denote the $\sigma$-field of all information gathered after completing Step (3.) of iteration $k$ of the block coordinate version of Algorithm 2, i.e., right before starting Step (4.) (the alternative direction step). Note that $\calF_k \subset \calG_k$. Applying \eqref{standard_smoothness} at Step (4.), we have deterministically:
\begin{align*}
F(x_{k+1}, y_{k+1}) &\leq F(\bar{x}_k, y_{k+1}) + \nabla F(\bar{x}_k, y_{k+1})^T((x_{k+1}, y_{k+1}) - (\bar{x}_k, y_{k+1})) + \tfrac{L_\nabla}{2}\|(x_{k+1}, y_{k+1}) - (\bar{x}_k, y_{k+1})\|^2 \\
&= F(\bar{x}_k, y_{k+1}) + \nabla_x F(\bar{x}_k, y_{k+1})^T(x_{k+1} - \bar{x}_k) + \tfrac{L_\nabla}{2}\|x_{k+1} - \bar{x}_k\|_X^2 \\
&= F(\bar{x}_k, y_{k+1}) + \sum_{i = 1}^N\left[\nabla_x^{(i)} F(\bar{x}_k, y_{k+1})^T(x_{k+1}^{(i)} - \bar{x}_k^{(i)}) + \tfrac{L_{\nabla}}{2}\|x_{k+1}^{(i)} - \bar{x}_k^{(i)}\|_{X,i}^2 \right] \\
&= F(\bar{x}_k, y_{k+1}) + \sum_{i = 1}^N\left[\bar{\beta}_k^i\nabla_x^{(i)} F(\bar{x}_k, y_{k+1})^Td_k^{(i)} + \tfrac{L_{\nabla}(\bar{\beta}_k^i)^2}{2}\|d_k^{(i)}\|_{X,i}^2 \right] \\
&\leq F(\bar{x}_k, y_{k+1}) + \sum_{i = 1}^N\left[\bar{\beta}_k^i\nabla_x^{(i)} F(\bar{x}_k, y_{k+1})^Td_k^{(i)} + \tfrac{L_{\nabla}\diam(S_i)^2(\bar{\beta}_k^i)^2}{2} \right] \\
&= F(\bar{x}_k, y_{k+1}) + \sum_{i = 1}^N\left[\bar{\beta}_k^i (\check{g}_k^{(i)})^Td_k^{(i)} + \bar{\beta}_k^i(\nabla_x^{(i)} F(\bar{x}_k, y_{k+1}) - \check{g}_k^{(i)})^Td_k^{(i)} + \tfrac{L_{\nabla}\diam(S_i)^2(\bar{\beta}_k^i)^2}{2} \right] \\
&= F(\bar{x}_k, y_{k+1}) + \sum_{i = 1}^N\left[-\bar{\beta}_k^i A_k^i + \bar{\beta}_k^i(\nabla_x^{(i)} F(\bar{x}_k, y_{k+1}) - \check{g}_k^{(i)})^Td_k^{(i)} + \tfrac{L_{\nabla}\diam(S_i)^2(\bar{\beta}_k^i)^2}{2} \right]
\end{align*}

Applying the inequality $(s^{(i)})^Tx^{(i)} \leq \tfrac{1}{2\gamma}\|s^{(i)}\|_{X\ast, i}^2 + \tfrac{\gamma}{2}\|x^{(i)}\|_{X,i}^2$ with $\gamma \gets L_\nabla$, $s \gets \nabla_x^{(i)} F(\bar{x}_k, y_{k+1}) - \check{g}_k^{(i)}$ and $x \gets \bar{\beta}_k^id_k^{(i)}$ yields:
\begin{align*}
F(x_{k+1}, y_{k+1}) &\leq F(\bar{x}_k, y_{k+1}) ~+ \\
& \ \ \ \ \ \ \ \sum_{i = 1}^N\left[-\bar{\beta}_k^i A_k^i + \tfrac{1}{2L_{\nabla}}\|\nabla_x^{(i)} F(\bar{x}_k, y_{k+1}) - \check{g}_k^{(i)}\|_{X\ast,i}^2 + \tfrac{L_{\nabla}(\bar{\beta}_k^i)^2}{2}\|d_k^{(i)}\|_{X,i}^2 + \tfrac{L_{\nabla}\diam(S_i)^2(\bar{\beta}_k^i)^2}{2} \right] \\
&\leq F(\bar{x}_k, y_{k+1}) + \sum_{i = 1}^N\left[-\bar{\beta}_k^i A_k^i + \tfrac{1}{2L_{\nabla}}\|\nabla_x^{(i)} F(\bar{x}_k, y_{k+1}) - \check{g}_k^{(i)}\|_{X\ast,i}^2 + \tfrac{\bar{C}_i(\bar{\beta}_k^i)^2}{2} \right] \\
&= F(\bar{x}_k, y_{k+1}) + \sum_{i = 1}^N\left[-\bar{\beta}_k^i A_k^i + \tfrac{\bar{C}_i(\bar{\beta}_k^i)^2}{2} \right] ~+ \tfrac{1}{2L_{\nabla}}\|\nabla_x F(\bar{x}_k, y_{k+1}) - \check{g}_k\|_{X\ast}^2 \ ,
\end{align*}
where the second inequality uses $\|d_k^{(i)}\|_{X,i} \leq \diam(S_i)$ and $\bar{C}_i \geq 2L_\nabla\cdot\diam(S_i)^2$. 

Notice that $\bar{\beta}_k^i = \min\left\{A_k^i/\bar{C}_i, \alpha^{\text{stop},i}_k\right\}$ minimizes the quadratic function $\beta \mapsto -\beta A_k^i + \tfrac{\bar{C}_i\beta^2}{2}$ on the interval $[0, \alpha^{\text{stop},i}_k]$. Hence, in particular we have that $- \bar{\beta}_k^iA_k^i + \tfrac{\bar{C}_i(\bar{\beta}_k^i)^2}{2} \leq 0$ and therefore:
\begin{equation*}
F(x_{k+1}, y_{k+1}) ~\leq~ F(\bar{x}_k, y_{k+1}) + \tfrac{1}{2L_{\nabla}}\|\nabla_x F(\bar{x}_k, y_{k+1}) - \check{g}_k\|_{X\ast}^2 \ .
\end{equation*}

The remainder of the proof is now exactly the same as that of Lemma \ref{fw-sd-lemma} starting at \eqref{therefore_zzz}.
\end{proof}

\subsection{Proof of Theorem 3.1}
Given Lemma \ref{block-alt-fw-sd-lemma}, the proof of Theorem 3.1 follows the exact same logic as that of Theorem 2.1 in Section \ref{big_theorem}.

\section{Additional numerical results on synthetic data}
In this set of experiments, we generated artificial data from a model that is described by an artificially generated sparse network. In particular, the network is composed of three layers of sizes $50 \times 50$, $50 \times 50$ and $50 \times 1$, and the activation functions are either ReLU or sigmoid (notice that sigmoid is smooth, but ReLU is not). We did not add bias terms for this experiment. For the first two layers of the true network, we randomly generated $m$ edges going into each node where $m \in \{5, 10, 15\}$ and the weight on each randomly sampled edge is either -1 or 1 with equal probability. All edges that are not selected have their weight equal to 0, and the last layer is fully connected.
We treated this as a regression problem with mean squared error loss, the feature matrix is composed of elements that are sampled i.i.d. from a standard Gaussian distribution, and the dependent variable values $y$ are chosen in a way to control the signal to noise ratio (SNR) in the set $SNR \in \{1, 5, 10\}$.
Our train, validation, and test sets are composed of 100,000, 20,000, and 100,000 data points, respectively.
For SFW and SFW-IF, the first two layers are treated as Frank-Wolfe layers and the last layer is fully dense and treated as an SGD layer.

For each combination of $m \in \{5, 10, 15\}$ and $SNR \in \{1, 5, 10\}$, we ran 30 trials over randomly generated true networks and datasets as described above. 
Figures 1 and 2 show box plots over these 30 trials, with three performance metrics of interest:  average number of non-zero edges (here average number of non-zeros per edge can be at most 50) going into each node in layers 1 and 2 (every value below 0.001 is considered to be 0), and the test set mean squared error. All Figures show that SFW and SFW-IF recover solutions that are sparser than SGD, which does not not promote this behaviour. SFW-IF is also able to consistently recover a solution that is sparser than SFW due to the incorporation of in-face directions (except for layer 2 using ReLU). Interestingly, using the sigmoid activation, the gap between the sparsity of SFW-IF and SFW is larger for layer 2 than layer 1; however, using the ReLU activation, this pattern is reversed.
Also, SFW-IF and SFW have comparable test set MSE to SGD, and for the sigmoid experiment they even have lower test MSE for the case of $m \in \{10, 15\}$. Figures 3 and 4 also consider a single instance with $m = 10$ and $SNR = 10$, using the sigmoid activiation, and display the evolution of the per layer average non-zeros and also the modified Frank-Wolfe gap. (Since we are only able to compute a stochastic estimate of the modified Frank-Wofle gap, we also display a smoothed version of this plot.) Notice that, throughout all iterations, SFW-IF consistently maintains a sparser solution than SFW. Interestingly, it appears that SFW-IF is also able to reduce the modified Frank-Wolfe gap faster than SFW.

\begin{figure}\label{fig:Sigmoid}
	\centering
	\begin{subfigure}{} 
		\includegraphics[width=\textwidth]{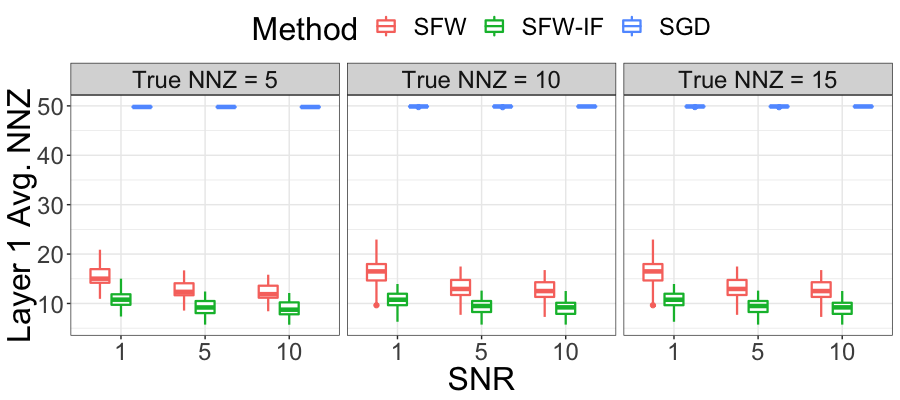}
	\end{subfigure}
	\vspace{1em} 
	\begin{subfigure}{} 
		\includegraphics[width=\textwidth]{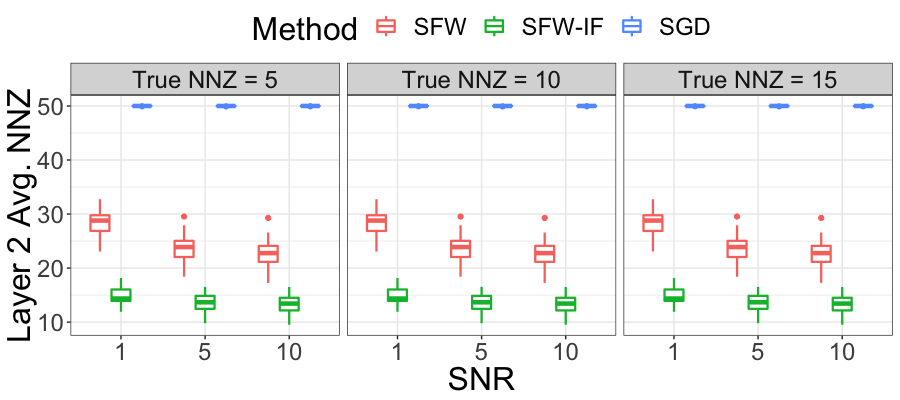}
	\end{subfigure}
	\begin{subfigure}{} 
		\includegraphics[width=\textwidth]{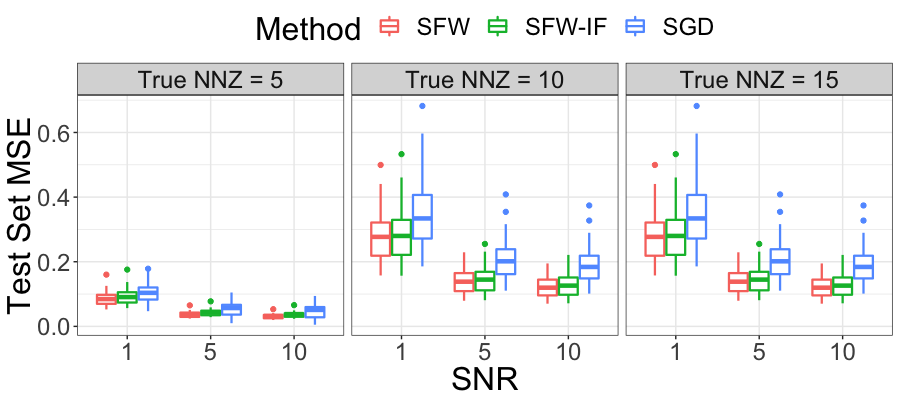}
	\end{subfigure}
	\caption{Results using sigmoid activation function} 
\end{figure}

\begin{figure}\label{fig:relu}
	\centering
	\begin{subfigure}{} 
		\includegraphics[width=\textwidth]{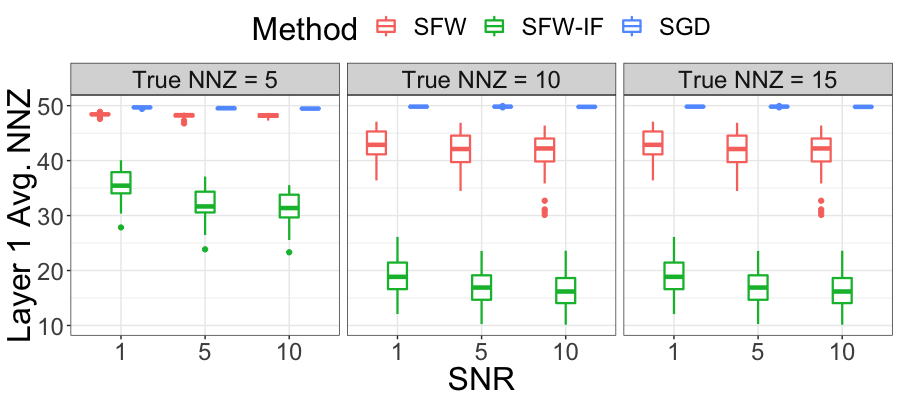}
	\end{subfigure}
	\vspace{1em} 
	\begin{subfigure}{} 
		\includegraphics[width=\textwidth]{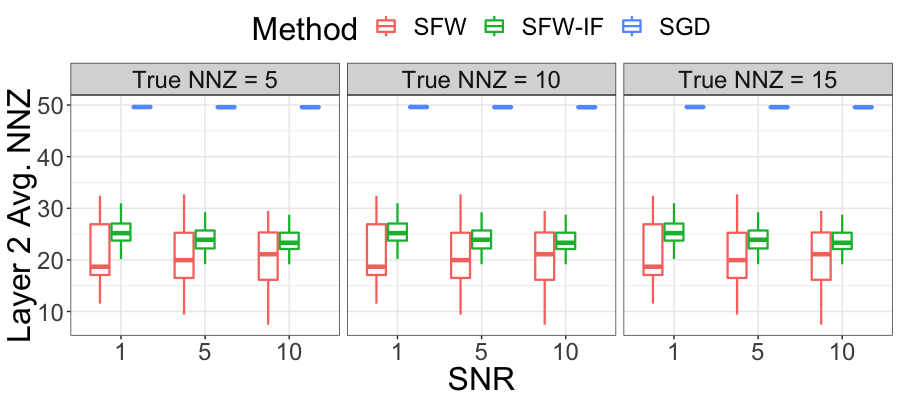}
	\end{subfigure}
	\begin{subfigure}{} 
		\includegraphics[width=\textwidth]{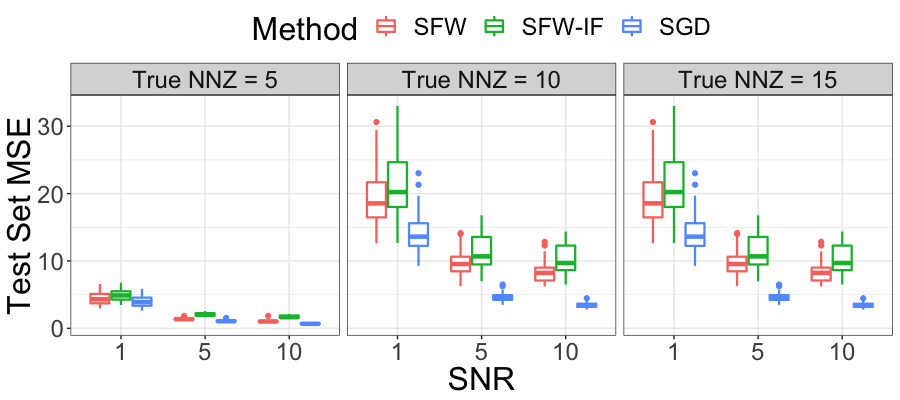}
	\end{subfigure}
	\caption{Results using ReLU activiation function} 
\end{figure}


\begin{figure}\label{fig:long1}
	\centering
	\begin{subfigure}{} 
		\includegraphics[width=0.75\textwidth]{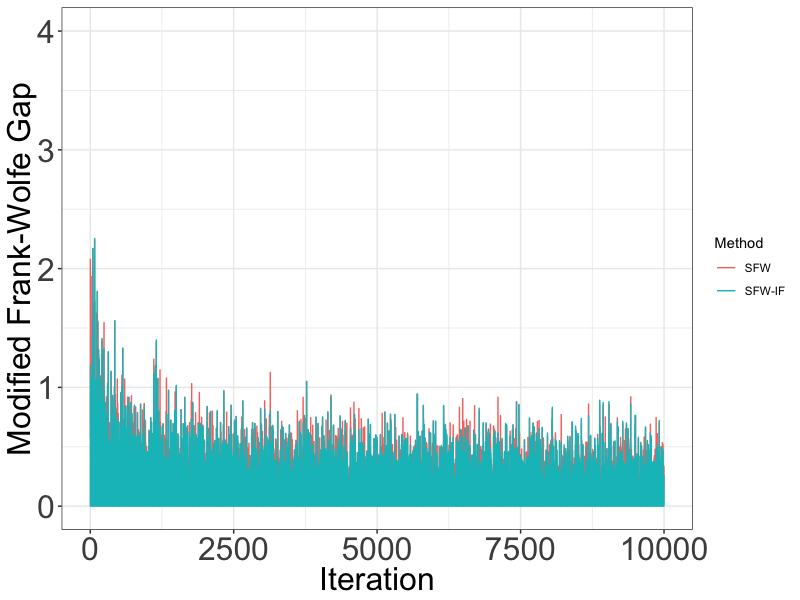}
	\end{subfigure}
	\vspace{1em} 
	\begin{subfigure}{} 
		\includegraphics[width=0.75\textwidth]{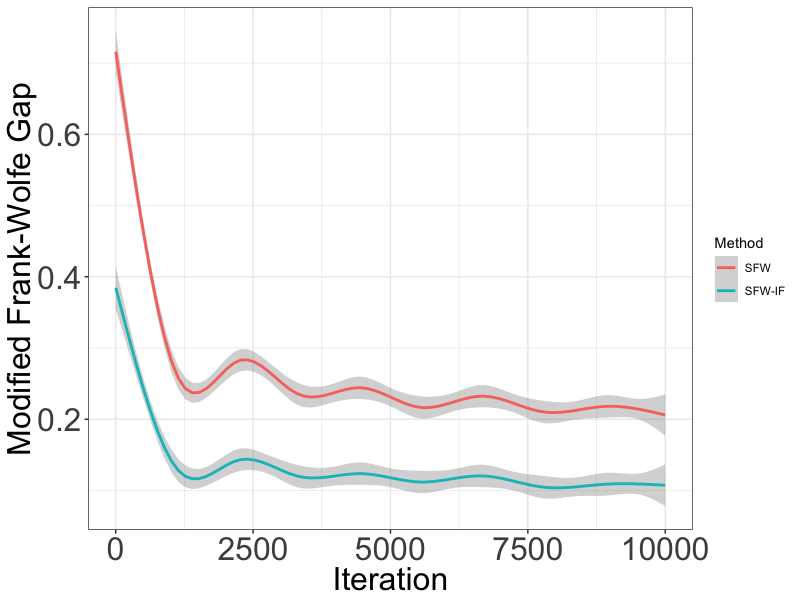}
	\end{subfigure}
	\caption{Modified Frank-Wolfe gap vs. iterations for SFW and SFW-IF, on an instance from a network generated using the sigmoid activation function with $m = 10$ non-zeros and $SNR = 10$.} 
\end{figure}

\begin{figure}\label{fig:long2}
	\centering
	\begin{subfigure}{} 
		\includegraphics[width=0.75\textwidth]{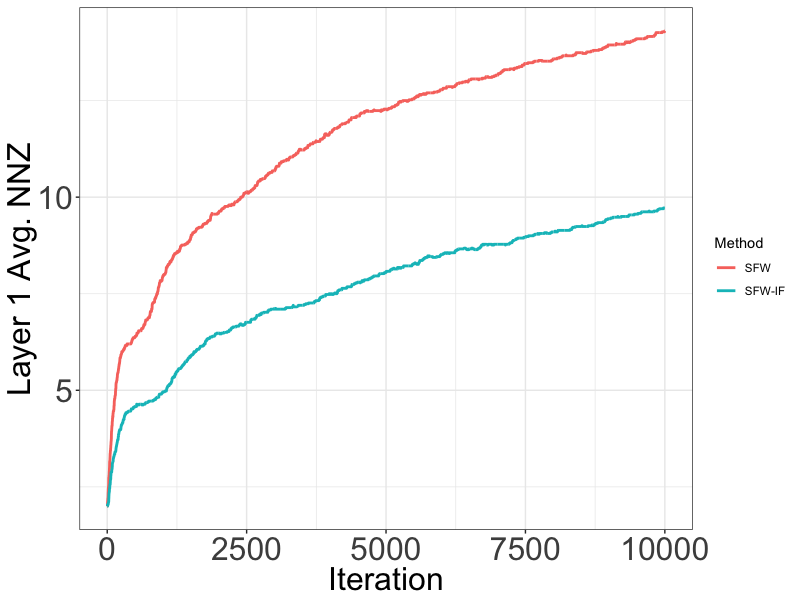}
	\end{subfigure}
	\vspace{1em} 
	\begin{subfigure}{} 
		\includegraphics[width=0.75\textwidth]{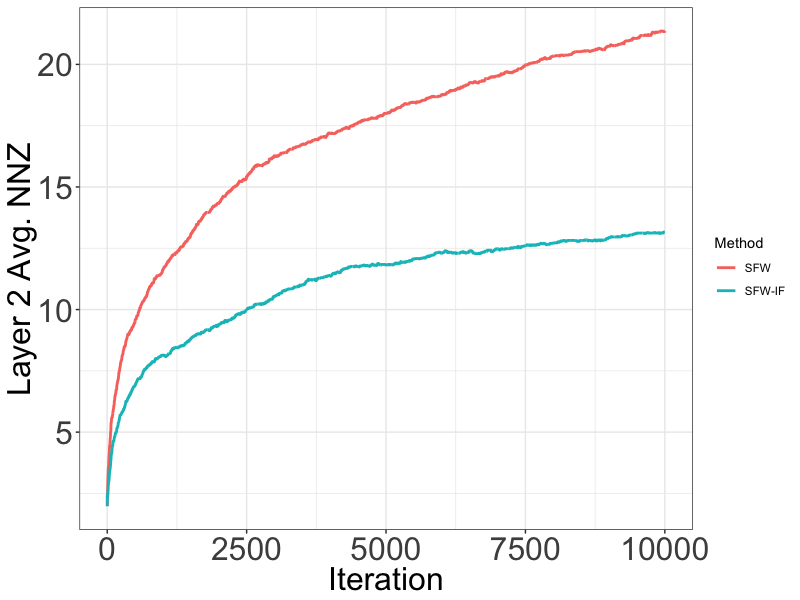}
	\end{subfigure}
	\caption{Average number of non-zeros (NNZ) per layer for SFW and SFW-IF, on an instance from a network generated using the sigmoid activation function with $m = 10$ non-zeros and $SNR = 10$.} 
\end{figure}

\end{document}